\documentclass[oneside,reqno,11pt,a4paper]{amsart}
\usepackage[T1]{fontenc}
\usepackage[margin=2.5cm]{geometry}
\usepackage{amsmath,amsfonts,amssymb,amscd,amsthm}
\usepackage[usenames,dvipsnames,svgnames]{xcolor}
\usepackage[utf8]{inputenc}
\usepackage{graphicx}
\graphicspath{{.}{figures/}}
\usepackage{caption}
\usepackage{subcaption}
\usepackage{adjustbox}
\usepackage[numbers,sort&compress]{natbib}
\usepackage{doi}
\usepackage{makecell}
\usepackage{hyperref}
\usepackage{acronym}
\usepackage{listings}
\usepackage{url}
\usepackage{color}
\usepackage{framed}
\usepackage{verbatim}
\usepackage{fancyvrb}
\usepackage{newtxtext}
\usepackage{newtxmath}
\usepackage{microtype} 
\usepackage[final]{pdfpages}
\usepackage{tikz}
\usepackage{booktabs}
\usepackage{multirow}

\hypersetup{
breaklinks=true,
bookmarksopen=true,
pdftitle={Space-time AgFEM},    
pdfauthor={Santiago Badia,Hridya Dilip,Francesc Verdugo},     
colorlinks=true,       
linkcolor=black,          
citecolor=blue,        
filecolor=black,      
urlcolor=blue           
}

\definecolor{bg}{rgb}{0.93,0.93,0.93}

\newtheorem{theorem}{Theorem}[section]
\newtheorem{lemma}[theorem]{Lemma}
\newtheorem{proposition}[theorem]{Proposition}

\newtheorem{remark}[theorem]{Remark}

\acrodef{pde}[PDE]{partial differential equation}
\acrodef{fe}[FE]{finite element}
\acrodef{fem}[FEM]{finite element method}
\acrodef{dof}[DOF]{degree of freedom}
\acrodefplural{dof}[DOFs]{degrees of freedom}
\acrodef{agfe}[AgFE]{aggregated finite element}
\acrodef{agfem}[AgFEM]{aggregated finite element method}
\acrodef{sfem}[StFEM]{standard finite element method}
\acrodef{xfem}[XFEM]{extended finite element method}
\acrodef{cg}[CG]{continuous Galerkin}
\acrodef{dg}[DG]{discontinuous Galerkin}
\acrodef{ale}[ALE]{arbitrary Lagrangian Eulerian}
\acrodef{jit}[JIT]{just-in-time}
\newcommand{\tnor}[1]{{\left\vert\kern-0.25ex\left\vert\kern-0.25ex\left\vert #1 
\right\vert\kern-0.25ex\right\vert\kern-0.25ex\right\vert}}
\definecolor{shadecolor}{gray}{.92}
\definecolor{incolor}{rgb}{0,0,.7}
\definecolor{outcolor}{rgb}{.65,0,0}
\definecolor{syntaxcolor}{rgb}{.65,0,0}

\begin{document}
\title[Space-time unfitted finite element method]{Space-time unfitted finite element methods for time-dependent problems on moving domains}
\author[S. Badia]{Santiago Badia$^{1,2,*}$}
\author[H. Dilip]{Hridya Dilip$^{1}$}
\author[F. Verdugo]{Francesc Verdugo$^{2}$}
\thanks{\null\
$^{1}$ School of Mathematics, Monash University, Clayton, Victoria, 3800, Australia.\
$^{2}$ Centre Internacional de M\`etodes Num\`erics a l'Enginyeria, Esteve Terrades 5, E-08860 Castelldefels, Spain.\
$^*$ Corresponding author.\
E-mails: {\tt santiago.badia@monash.edu} (SB),
{\tt hridya.dilip@monash.edu} (HD),
{\tt fverdugo@cimne.upc.edu} (FV)
}

\date{\today}
\begin{abstract}
  We propose a space-time scheme that combines an unfitted finite element method in space with a discontinuous Galerkin time discretisation for the accurate numerical approximation of parabolic problems with moving domains or interfaces. We make use of an aggregated finite element space to attain robustness with respect to the cut locations. The aggregation is performed slab-wise to have a tensor product structure of the space-time discrete space, which is required in the numerical analysis. {As an aternative, we also propose a space-time ghost penalty stabilisation term to attain robustness.} We analyse the proposed algorithm, providing stability, condition number bounds and anisotropic \emph{a priori} error estimates. A set of numerical experiments confirm the theoretical results for a parabolic problem on a moving domain. The method is applied for a mass transfer problem with changing topology.
\end{abstract}
\maketitle

\noindent{{\bf {Keywords}}: Embedded methods; unfitted finite elements; space-time discretisations.

\newcommand{\jump}[1]{{\llbracket #1 \rrbracket }}
\newcommand{\sbcom}[1]{{\color{PineGreen}{***{SB: #1}***}}}
\newcommand{\hdcom}[1]{{\color{olive}{{HD: #1}}}}
\newcommand{\fvcom}[1]{{\color{magenta}{*{FV: #1}*}}}
\newcommand{\modify}[1]{{\color{red}{#1}}}
\section{Introduction}\label{sec:introduction}

\par Numerical simulations using standard \acp{fem} require the generation of body-fitted meshes, which is one of the main bottlenecks of the simulation workflow. This problem is exacerbated in applications that involve moving interfaces and evolving geometries. The method of lines, which discretises in space and time separately, cannot be readily applied to transient problems with moving domains or interfaces since it assumes a constant geometry in time. In order to solve this problem, one can consider \ac{ale} schemes \cite{Donea1982,nobile1999stability}. \ac{ale} schemes require frequent remeshing and are not suitable for large geometrical variations or topological changes. Another approach is to use variational space-time formulations on space-time body-fitted meshes. These methods have been widely used in applications like fluid-structure interaction \cite{Thompson1996,LeBeau1993,Tezduyar2006}. Even though variational space-time schemes can be applied to moving domains/interfaces, they do require space-time meshes, which are unfeasible in general. The mathematical analysis of space-time methods has been considered, e.g., in \cite{Saito2020} (for a \ac{dg} method in time for parabolic equations on body-fitted domains and constant geometries) and in \cite{Sudirham2006} (for a space-time \ac{dg} method for advection-diffusion on time-dependent domains). 

\par Unfitted (also known as \emph{immersed} or \emph{embedded}) \ac{fe} formulations lower the geometrical requirements since they do not require body-fitted meshes but simple, e.g., Cartesian, background meshes. Hence, unfitted \ac{fem} are becoming increasingly popular in applications with moving interfaces such as fluid-structure interactions \cite{Formaggia2021,Burman2014,Schott2019}, fracture mechanics \cite{Giovanardi2017,Dekker2019xfem}, and in applications with changing geometries such as additive manufacturing \cite{Neiva2020,Carraturo2020additive} and stochastic geometry problems \cite{Badia2021monte}.  In \cite{Lehrenfeld:462743}, a mass transport problem across an evolving interface is analysed using a variational space-time \ac{dg} \ac{xfem}.

\par However, unfitted \acp{fem} are prone to ill-conditioning problems when dealing with unfitted boundaries and high contrast interface problems \cite{Reusken2014,Badia2018aggregated,Burman2014cutfem}. If the intersection of a cut background cell with the physical domain is small, it can lead to a so-called \emph{small cut cell problem}. The support of the \ac{fe} shape functions corresponding to the background cell can have an arbitrarily small support, leading to almost singular system matrices. Several methods \cite{Burman2014cutfem,Kummer2016,Lehrenfeld2016,Guzmn2017,Li2019} have been proposed to circumvent the small cut cell problem. However, only few formulations are robust and optimal with respect to the cut cell position. Some methods include additional terms that enhance the stability of the \ac{fe} discretisation while keeping optimal convergence (see, e.g., the \emph{ghost penalty} \cite{Burman2010ghost} formulation used in Cut\ac{fem} \cite{Burman2014cutfem,Zahedi2017,Claus2019cutfem}). Another approach, used in this work, involves cell \emph{aggregation} (or \emph{agglomeration}) techniques. These techniques can readily be applied to numerical methods that can handle general polytopal meshes, e.g., \ac{dg} or hybridisable methods (see, e.g., \cite{Bassi2012agglomeration,Saye2017,Engwer2012,Burman2021}). Cell aggregation for $\mathcal{C}^0$ \ac{fe} spaces has been proposed in \cite{Badia2018aggregated}, where it was coined \ac{agfem}.

\par In \ac{agfem}, the \acp{dof} associated to \ac{fe} functions that have arbitrarily small support and can lead to ill-conditioning are eliminated. This is attained by designing a discrete extension operator that constrains the ill-posed \acp{dof} using the well-posed \acp{dof} while preserving $\mathcal{C}^0$ continuity. \ac{agfem} enjoys good numerical properties, such as stability, bounded condition numbers, optimal convergence and continuity with respect to data; detailed mathematical analysis of this method is included in \cite{Badia2018aggregated} for elliptic problems, in \cite{Badia2018mixed} for the Stokes equation and in \cite{Badia2022robust} for higher-order \acp{fe}. The method is also amenable to arbitrarily complex 3D geometries \cite{Martorell}, distributed implementations for large scale problems \cite{Verdugo2019}, error-driven $h$-adaptivity and parallel tree-based meshes \cite{Badia2021parallel}, explicit time-stepping for the wave equation \cite{Burman2020explicit}  and elliptic interface problems with high contrast \cite{Neiva2021}. A \emph{weak} \ac{agfem} technique is proposed in \cite{Badia2022ghost}, which is much less sensitive to stabilisation parameters than the \emph{ghost penalty} method. 

\par Despite the potential of unfitted \acp{fem} for transient problems with moving boundaries and interfaces, few formulations are robust and enjoy optimal convergence. The space-time \ac{dg} \ac{xfem} scheme proposed in \cite{Lehrenfeld:462743} is not robust to the cut location and the error estimate is suboptimal with respect to time. The main reason for the suboptimal error estimates is the fact that the \ac{fe} space cannot be expressed as a slab-wise tensor product in space-time. The Cut\ac{fem} formulation in \cite{Zahedi2017} makes use of {space-time quadratures to approximate moving domains} for transient convection-diffusion problems. 
Robust and optimally convergent space-time \ac{dg} formulations are presented in \cite{preuss,heimann_ma,preuss_ma}, in which the robustness of these methods are due to additional stabilisation terms in the weak formulation. Space-time Cut\ac{fem} on overlapping meshes with optimal convergence are explored in \cite{lundholm2015space,lundholm2021cut}.

\par The novelties of this work are the following:
\begin{enumerate}
  \item We propose a novel unfitted variational space-time formulation on moving domains/interfaces that is robust with respect to the small cut cell problem. Robustness is attained by extending \ac{agfem} to space-time. The spatial discretisation can handle both continuous (nodal) and discontinuous \ac{fe} spaces, while a \ac{dg} space is used in time.
  \item We carry out a detailed mathematical analysis proving that this method enjoys sought-after numerical properties such as well-posedness, stability, bounded condition numbers and optimal convergence. In addition, we provide implementation details and perform a set of numerical experiments that support the theoretical results.
\end{enumerate}
 In particular, we consider a slab-wise cell aggregation scheme and a space-time discrete extension operator that can be expressed as a space-time tensor product. This way, the \ac{agfe} space is constant at each time slab, and we can prove optimal error bounds.

\par The outline of this work is as follows. First, we introduce  the embedded geometry setup, the aggregation strategy and construct the \ac{agfe} spaces in Section~\ref{sec:st_agfem}. In Section~\ref{sec:approximation}, the proposed space-time \ac{agfem} discretisation is introduced for a model problem. We perform the numerical analysis of the method in Section~\ref{sec:numerical_analysis} and numerical experiments that support these results are presented in Section~\ref{sec:num_expts}. Finally, we draw some conclusions in Section~\ref{sec:conclusion}.

\section{Space-time aggregated finite element method}\label{sec:st_agfem}

\subsection{Embedded geometry setup}\label{sec:geo_setup}

\par In this section, we provide a set of geometrical definitions that will be required to define the proposed formulation. We refer to Figure~\ref{fig:embed_bdry_setup} for an illustration of many of the definitions below. 

Let us consider an open, bounded, connected Lipschitz domain $\Omega^0 \subset \mathbb{R}^d$, with $d \in \{2,3\}$ the number of spatial dimensions, a time domain $[0,T]$ and a smooth diffeomorphism \mbox{$\boldsymbol{\varphi}_t(\boldsymbol{x}): \Omega^0 \rightarrow \mathbb{R}^{d}$} for any $t \in [0,T]$. We define $\Omega(t) \doteq  
\left\{ \boldsymbol{\varphi}_t(\boldsymbol{x}) \ : \ \boldsymbol{x} \in \Omega^0 \right\}$ (the domain at a given time step) and $Q = \left\{ \boldsymbol{x} \in \Omega(t) \ : \ t \in [0,T] \right\}$ (the space-time domain). For simplicity, we assume that $Q$ is a polytopal domain. 
$\partial \Omega(t)$ represents the boundary of $\Omega(t)$. We consider a partition of the boundary into $\partial \Omega_D(t)$ and $\partial \Omega_N(t)$, the Dirichlet and Neumann spatial boundaries, resp. Thus, $\partial \Omega(t) = \partial \Omega_D(t) \cup\partial \Omega_N(t)$ and $\partial \Omega_D(t) \cap\partial \Omega_N(t) = \emptyset$. The Dirichlet and Neumann boundaries of the space-time domain are $\partial Q_D \doteq \cup_{t \in (0,T)}\partial \Omega_D(t) \times \{t\}$ and $\partial Q_N\doteq \cup_{t \in (0,T)}\partial \Omega_N(t) \times \{t\}$, resp.  The boundary of $Q$ is $\partial Q \doteq \Omega(0) \cup \Omega(T) \cup  \partial Q_D \cup \partial Q_N$.

\par Let us define a spatial artificial domain $\Omega_{art} \subset \mathbb{R}^d$ such that $\Omega(t) \subset \Omega_{art}$ for all $t\in [0,T]$. One can consider a simple geometry for $\Omega_{art}$, e.g., a bounding box, which can be meshed using a Cartesian grid. We can also define the space-time artificial domain $Q_{art} \doteq \Omega_{art} \times [0,T]$, such that $Q \subset Q_{art}$.

\par Let $0 = t^0 < t^1 < \dots < t^N = T$ and  
$J^n \doteq (t^{n-1},t^n), 1 \le n \le N$, denote the $n$-th time slab. $\left\{ J^n \right\}_{n=1}^{N}$ is a partition of $[0,T]$. The size of each time slab $J^n$ is denoted by $\tau^n \doteq t^n - t^{n-1}$ (the so-called time step size) and $\tau \doteq \max_{n=1,\ldots,N} \tau^n$. The artificial domain and the space-time domain corresponding to a time slab $J^n$ are denoted as $Q^n_{art} \doteq \Omega_{art} \times J^n$ and $Q^n \doteq  \cup_{t\in J^n} \Omega(t)\times \{t\}$, resp. Furthermore, the intersection of the Dirichlet and Neumann boundaries with $Q^n_{art}$ are denoted as $\partial Q^n_D$ and $\partial Q^n_N$, resp. We also use the notation $\Omega^n \doteq \Omega(t^n)$, $n=0,\ldots,N$.  

\par Let $\bar{\mathcal{T}}_{h,art}^n$ be a \emph{conforming}, shape regular and quasi-uniform partition of $\Omega_{art}$. The space-time mesh $\mathcal{T}_{h,art}^n$ is the Cartesian product of $\bar{\mathcal{T}}_{h,art}^n$ and $J^n$, i.e.,
\[
\mathcal{T}_{h,art}^n \doteq \{ \bar{T} \times J^n \in \mathcal{T}_{h,art} \ : \ \bar{T} \in \bar{\mathcal{T}}_{h,art}^n \}.
\] 
The super-index $n$ in $\bar{\mathcal{T}}_{h,art}^n$  stands for the fact that the background mesh can be different at different time slabs. E.g., $\bar{\mathcal{T}}_{h,art}^n$ can be a background $n$-tree mesh with adaptive mesh refinement. We use the notation $T^n$ for cells in $\mathcal{T}_{h,art}^n$  and $\bar{T}^n$ for cells in $\bar{\mathcal{T}}_{h,art}^n$. By construction, we can define the injective map $\bar{\cdot}: \mathcal{T}_{h,art}^{n} \longrightarrow \bar{\mathcal{T}}_{h,art}^{n}$ such that $T^n = \bar{T}^n \times J^n$, i.e., a map from space-time to space-only cells at each time slab. In the analysis, we assume that $\bar{\mathcal{T}}_{h,art}^n$ is a shape-regular and quasi-uniform mesh with characteristic cell size $h$. 

\begin{figure}[ht!]
\includegraphics[width=0.99\textwidth]{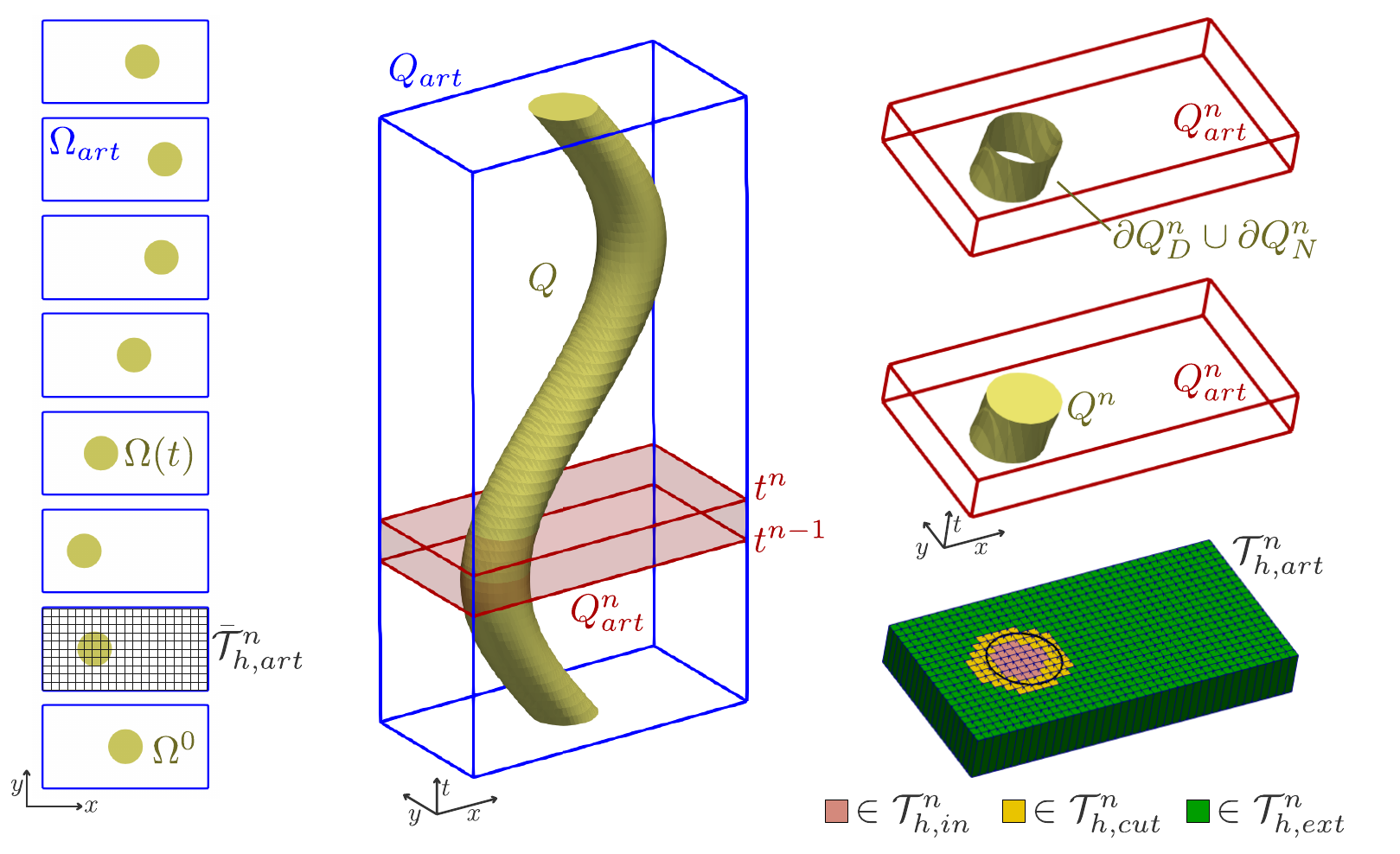}
\caption{Graphical representation of the main geometrical quantities associated with the space-time embedded finite element setup for a 2D+1D example. We note in the figure at the bottom-right corner that some cells that are not cut on $t^n$ appear as cut. This is because these cells are cut at some time value $t \in [t^{n-1},t^n]$.}
\label{fig:embed_bdry_setup}
\end{figure}

\definecolor{FigCyan}{RGB}{91,225,214}
\definecolor{FigYellow}{RGB}{242,225,33}

\subsection{Cell aggregation}\label{sec:aggregation}

\par The direct use of unfitted \acp{fem} on the previously defined meshes is not robust with respect to cut locations. As commented in the introduction, one could consider using stabilisation techniques to remedy this problem. Another approach, which is followed in this work, relies on the definition of aggregated or agglomerated meshes. In particular, cells are aggregated in such a way that all aggregates have a \emph{large enough} portion inside the physical domain, e.g., there is one internal cell per aggregate.   

\par We refer to \cite{Badia2018aggregated} for the aggregation strategy required in the space-only case. In space-time, we apply this algorithm slab-wise to prevent cells at different time slabs to be merged to form aggregates. That would complicate the implementation and numerical analysis and have a serious impact on the computational cost. Our motivation is to end up with a space-time solver that only requires a set of sequential slab-wise solvers, as time marching methods. 

\par The aggregation algorithm requires a classification of active cells between well-posed and ill-posed cells. The most straightforward definition is to classify interior cells as well-posed and cut cells as ill-posed. If ${T}^n \subset Q^n$, then ${T}^n$ is an internal cell. If ${T}^n \cap Q^n = \emptyset$, then ${T}^n$ is an external cell. Otherwise, ${T}^n$ is a cut cell. The set of internal, external and cut cells on time slab $J^n$ are denoted as $\mathcal{T}_{h,in}^n$, $\mathcal{T}_{h,ext}^n$ and $\mathcal{T}_{h,cut}^n $, resp. (see Figure~\ref{fig:embed_bdry_setup}). The union of internal, external and cut cells on each time slab is denoted as $Q_{in}^n \subset Q$, $Q_{ext}^n$ and $Q_{cut}^n$, resp. We define the set of active cells as $\mathcal{T}^n_{h,act}\doteq \mathcal{T}^n_{h,in} \cup \mathcal{T}^n_{h,cut}$ and their union as $Q^n_{act}$. We can readily define the space-only meshes $\bar{\mathcal{T}}^n_{h,act}\doteq \bar{\mathcal{T}}^n_{h,in} \cup \bar{\mathcal{T}}^n_{h,cut}$ using the map $\bar{\cdot}$ over the cells of the respective space-time meshes. The union of cells in $\bar{\mathcal{T}}_{h,act}^n$ is represented with $\Omega_{act}^n$.

This definition can be further refined by considering a numerical parameter $\eta_0 > 0$. Given a time slab $J^n$, one can compute the cell-wise quantity
\[
 \eta^n(T^n) = \underset{t \in J^n}{\mathrm{min}}\frac{ | \bar{T}^n \cap \Omega(t)|}{ |\bar{T}^n|}, \qquad \forall T^n \in \mathcal{T}_{h,act}^n. 
\] 
If $\eta^n(T^n) \ge \eta_0$, then $T^n$ is a well-posed cell. Otherwise, it is ill-posed. Note that this definition enforces that any space-time cell must have a significant portion in the spatial domain at all times, thus it is \emph{anisotropic}. When $\eta_0 = 1$, then well-posed cells are internal cells and ill-posed cells are cut cells. For brevity, we will consider this case in the following exposition, even though the general case does not involve any modification.

\par Next, at each time slab, the aggregation strategy introduced in \cite{Badia2018aggregated} (in a spatial mesh only) is performed on the active mesh $\mathcal{T}_{h,act}^n$. Very briefly, the algorithm performs the following steps: {(1) well-posed cells are marked as touched first;} (2) each ill-posed cell that is neighbour\footnote{We recall that neighbour means neighbour in space. This can be achieved by modifying the aggregation strategy or by using the standard space aggregation verbatim at each slab independently.} of touched cells is merged to one of these and marked as touched; (3) repeat (2) till all active cells are touched. This algorithm returns a set of aggregates $\mathrm{ag}(\mathcal{T}_{h,act}^n)$  that contain one and only one well-posed cell. This well-posed cell is called the \emph{root} cell of the aggregate. We refer to \cite{Badia2018aggregated} for more details about the aggregation strategy in space, e.g., bounds for the size of the resulting aggregates.

\subsection{Space-time unfitted \ac{fe} spaces}\label{sec:unf_fem_space}

\par Our aim is to construct \ac{agfe} spaces on each time slab making use of the aggregated meshes defined above. We start by introducing some notations. It is crucial to note that the \ac{agfe} spaces on each time slab may be different, even if $\mathcal{T}_{h,art}^{n}$ is the same for all slabs, due to the evolving geometry in time. 

\par Let $T^n = \bar{T}^n \times J^n \in \mathcal{T}_{h,act}^{n}$, where $\bar{T}^n \in \bar{\mathcal{T}}_{h,act}^n$. We define the local \ac{fe} space as a tensor product of spatial and temporal polynomials. For $d-$simplex spatial meshes, the local \ac{fe} space ${V}({{T}^n}) \doteq \mathcal{P}_p({\bar{T}^n}) \otimes \mathcal{P}_q({J^n})$ is the space of polynomials of order less than or equal to $p$ in the spatial variables $x_1,x_2,\dots,x_d$ and polynomials of order less than or equal to $q$ in the temporal variable. For $d-$cube spatial meshes the local \ac{fe} space ${V}({{T}^n}) \doteq \mathcal{Q}_p({\bar{T}^n}) \otimes \mathcal{P}_q({J^n})$ is the space of polynomials of order less than or equal to $p$ in each of the spatial variables $x_1,x_2,\dots,x_d$ and polynomials of order less than or equal to $q$ in the temporal variable.

\par In this work, we restrict ourselves to Lagrangian \ac{fe} methods. Observe that the basis of the local \ac{fe} space $V({{T}^n})$ is the tensor product of the Lagrangian basis of order $p$ in space and a basis for univariate polynomials of order $q$ in time. (The choice of a basis in time is flexible, since we will not enforce $\mathcal{C}^0$ continuity in time, but we will consider a Lagrangian basis for simplicity.) Let $\mathcal{N}({{T}^n})$ denote the set of Lagrangian nodes of ${{T}^n}$; any  $a \in \mathcal{N}({{T}^n})$ can be expressed as a tuple $(a_{\boldsymbol{x}},a_t)$ of space and time nodes. The dual basis of \acp{dof} corresponds to the pointwise evaluation at these nodes. Analogously, the space-time shape functions associated to node $b \in \mathcal{N}({{T}^n})$  can be expressed as $\Phi^b(\boldsymbol{x},t) \doteq \phi^{b_{\boldsymbol{x}}}(\boldsymbol{x})\varphi^{b_t}(t)$, i.e., the tensor product of a spatial shape function $\phi^{b_{\boldsymbol{x}}}(\boldsymbol{x})$ and temporal shape function $\varphi^{b_t}(t)$. It satisfies $\Phi^b({\boldsymbol{x}}^{a_{\boldsymbol{x}}},t^{a_t}) = \phi^{b_{\boldsymbol{x}}}({\boldsymbol{x}}^{a_{\boldsymbol{x}}})\varphi^{b_t}(t^{a_t}) = \delta_{a_{\boldsymbol{x}}b_{\boldsymbol{x}}}\delta_{a_tb_t}$, where ${\boldsymbol{x}}^{a_{\boldsymbol{x}}}$ and $t^{a_t}$ are the spatial and temporal coordinates of the node $a$,  resp., and $\delta$ is the Kronecker delta.

\par We consider a \ac{cg} global \ac{fe} space in the spatial direction at each time slab $J^n$. For \ac{cg} methods, this is attained (on conforming meshes) by a local-to-global \ac{dof} map such that the resulting global space of functions is $\mathcal{C}^0$ continuous. To this end, at time slab $J^n$, we introduce the active space-time \ac{fe} space 
\[
V_{h,act}^n \doteq \{v \in \mathcal{C}^0(Q^n_{act}) : v|_{T} \in V(T), 
\text{ for any } T \in \mathcal{T}_{h,act}^{n} \}.
\]
It can also be defined as a tensor-product space as follows. We define a space-only global \ac{fe} space on $\Omega_{act}^n$ as
\[
\bar{V}_{h,act}^n =\{v \in \mathcal{C}^0(\Omega_{act}^n) : v|_{\bar{T}} \in \mathcal{X}_p(\bar{T}), 
\text{ for any } \bar{T} \in \bar{\mathcal{T}}_{h,act}^n \},
\]
where $\mathcal{X}_p(\bar{T})$ is $\mathcal{P}_p(\bar{T})$ for simplicial meshes and $\mathcal{Q}_p(\bar{T})$ for hexahedral meshes. With all these ingredients, $V_{h,act}^n$ can be equivalently defined as $V_{h,act}^n \doteq \bar{V}_{h,act}^n \otimes \mathcal{P}_q(J^n)$. We can proceed analogously for interior cells to define $V_{h,in}^n \doteq \bar{V}_{h,in}^n \otimes \mathcal{P}_q(J^n)$.

\par Since we consider a \ac{dg} approximation in time, the global space-time space can readily be defined as a Cartesian product of the above defined slab-wise spaces. In any case, the global problem is \emph{never solved at once} but sequentially slab-by-slab.

\subsection{Space-time \ac{agfe} spaces}\label{sec:agfem_space}

\par It has been established that solving the \ac{fe} problem on the active space leads to ill-conditioning problems. Therefore, we use an \ac{agfe} space to solve this issue. The key idea of \ac{agfem} is to eliminate the problematic \acp{dof} using well-posed \acp{dof}. This is achieved by defining a discrete extension operator. Let $\bar{\mathcal{E}} : \bar{V}_{h,in} \rightarrow \bar{V}_{h,act}$ be the spatial extension operator between space-only interior $\bar{V}_{h,in}$ and active $\bar{V}_{h,act}$ \ac{fe} spaces. The space-only discrete extension operator has been previously described, e.g., in \cite{Badia2018aggregated}. The spatial \ac{agfe} space is defined as $\bar{V}_{h,ag} \doteq \bar{\mathcal{E}}(\bar{V}_{h,in})$. The discrete extension operator relies on a set of linear constraints that constrain the ill-posed \acp{dof} (i.e., the ones that only belong to ill-posed cells) by the well-posed \acp{dof} (the ones that belong to at least one well-posed cell). We do not provide the full construction of the space-only operator for the sake of conciseness. The interested reader can find the complete definition in \cite{Badia2018aggregated} for the case of conforming background meshes, \cite{Badia2021parallel} for non-conforming adaptive meshes and \cite{Badia2022robust} for a version of this extension that is well-suited to high-order approximations. The space-time extension proposed herein can be applied in all these situations.

\par In this work, we must define a slab-wise space-time discrete extension operator between $V_{h,in}^n$ and $V_{h,act}^n$. We do this by combining the slab-wise aggregation algorithm in Section \ref{sec:aggregation}, the tensor-product definition of these spaces in Section \ref{sec:unf_fem_space} and the space-only extension operator in \cite{Badia2018aggregated}. In particular, at time slab $J^n$, the space-time extension operator $\mathcal{E}^n:V_{h,in}^n \rightarrow V_{h,act}^n$ is defined as follows:
\begin{equation}\label{eq:extension_tensor_product}
  \mathcal{E}^n(V_{h,in}^n) = \mathcal{E}^n(\bar{V}_{h,in}^n \otimes \mathcal{P}_q(J^n)) = \bar{\mathcal{E}}^n(\bar{V}_{h,in}^n) \otimes \mathcal{P}_q(J^n),
\end{equation}
where $\bar{\mathcal{E}}^n: \bar{V}_{h,in}^n \rightarrow \bar{V}_{h,act}^n$ is a space-only extension operator at slab $J^n$. The space-time aggregated \ac{fe} space on the time slab, $V^n_{h,ag} \doteq \mathcal{E}^n(V_{h,in}^n)$. The global space-time \ac{agfe} space is \mbox{$\mathcal{V}_{h,ag} \doteq V^1_{h,ag}\times\dots\times V^N_{h,ag}$}. By construction, we have $\mathcal{E}^n(u_{h})(\boldsymbol{x},t) = \bar{\mathcal{E}}^n(u_{h}(\cdot,t))(\boldsymbol{x})$ for $u_h \in V_{h,in}^n$. 

\section{Approximation of parabolic problems on moving domains}\label{sec:approximation}

\subsection{Model problem}\label{sec:model_problem}

\par We start by introducing some anisotropic functional spaces. Let \mbox{$\boldsymbol{\alpha} = (\alpha_1,\alpha_2,\dots,\alpha_d)$} and $|\boldsymbol{\alpha}|=\alpha_1 + \dots + \alpha_d$. We define the anisotropic Sobolev space of order $(s_s,s_t)$ on a domain $\mathcal{D} \subset \mathbb{R}^{d+1}$  as
\begin{equation}\label{eq:anisotropic_Sobolev_space}
	H^{(s_s,s_t)}(\mathcal{D}) \doteq \{u \in L^2(\mathcal{D}) : \partial^{\alpha_1}_{x_1}\dots\partial^{\alpha_d}_{x_d} u, \, \partial^\beta_t u \in L^2(\mathcal{D}) \text{ for } |\boldsymbol{\alpha}| \le s_s, \beta \le s_t  \}.
\end{equation}
The norm and semi-norm associated with this Sobolev space are:
\begin{equation}\label{eq:anisotropic_sobolev_norms}
	\begin{aligned}
		\|v\|_{H^{(s_s,s_t)}(\mathcal{D})}^2 & \doteq 
      \sum \limits_{\substack{|\boldsymbol{\alpha}|\le s_s}} \left\| \partial^{\alpha_1}_{x_1}\dots\partial^{\alpha_d}_{x_d}  v \right\|^2_{L^2(\mathcal{D})} +
      \sum \limits_{\substack{\beta \le s_t}} \left\| \partial^\beta_t v \right\|^2_{L^2(\mathcal{D})}, \\	
      |v|_{H^{(s_s,s_t)}(\mathcal{D})}^2 & \doteq 
      \sum \limits_{\substack{|\boldsymbol{\alpha}| = s_s}} \left\| \partial^{\alpha_1}_{x_1}\dots\partial^{\alpha_d}_{x_d}  v \right\|^2_{L^2(\mathcal{D})} + \left\| \partial^{s_t} v \right\|^2_{L^2(\mathcal{D})}.
	\end{aligned}
\end{equation}

\par We introduce the proposed formulation for the convection-diffusion equation on moving domains with non-homogeneous boundary conditions as a model problem. In any case, the proposed methodology can be extended to other parabolic problems. Using the ideas in \cite{Badia2018mixed}, it can also be generalised to indefinite systems, e.g., incompressible flows. 

\par We represent with $H_{g_D}^{(1,s_t)}(Q)$ the subspace of function in $H^{(1,s_t)}(Q)$ with trace 
$g_D \in H^{(1/2,s_t)}(\partial Q_D)$ on the Dirichlet boundary. The model problem seeks to find $u:Q \rightarrow \mathbb{R} \in H^{(1,0)}_{g_D}(Q) \cap H^{(-1,1)}(Q)$ such that
\begin{equation}
	\label{eq:heat_eqn}
	\begin{aligned}
		&\partial_t u + \boldsymbol{w} \cdot \boldsymbol{\nabla} u - \mu \Delta u = f &\text{ in } & H^{(-1,0)}(Q), \qquad
		 &\mu \nabla u  \cdot \boldsymbol{n} = g_N &\text{ in } & H^{(-1/2,0)}(\partial Q_N),
	\end{aligned}
\end{equation}
and the initial condition $u(x,0) = u^0(x) \in L^2(\Omega^0)$, where $\mu$ is the diffusion coefficient, the source term $f \in H^{(-1,0)}(Q)$, $\boldsymbol{w} \in L^\infty(Q)$ is a solenoidal convective field,  the boundary flux $g_N \in H^{(-1/2,0)}(\partial Q_N)$ and $\boldsymbol{n}$ denotes the outward normal to the boundary $\partial Q_N$. The weak form of \eqref{eq:heat_eqn} consists in finding 
\begin{equation}
\label{eq:weak_problem_st}
u \in H_{g_D}^{(1,0)}(Q) \cap H^{(-1,1)}(Q) \quad : \quad B(u,v) = L(v) \qquad \forall v \in H^{(1,0)}(Q),
\end{equation}
where
\begin{equation} 	\label{eq:global_weak_problem}
	\begin{aligned}
	B(u,v) & = \int_Q \partial_t u \mathrm{~} v \mathrm{~d}\boldsymbol{x} \mathrm{~d}t + \int_0^T a(t,u,v)  \mathrm{~d}t  + \int_0^T c(t,w,u,v)  \mathrm{~d}t, \\
	a(t,u,v)  &= \int_{\Omega(t)} \mu \nabla u \cdot \nabla v \mathrm{~d}\boldsymbol{x},  \quad 
c(t,w,u,v)  = \int_{\Omega(t)} \boldsymbol{w} \cdot \boldsymbol{\nabla} u \mathrm{~} v \mathrm{~d}\boldsymbol{x},  \\
	L(v)&= \int_0^T l(t,v) \mathrm{~d}t, \quad
	l(t,v)= \int_{\Omega(t)} f \mathrm{~} v \mathrm{~d}\boldsymbol{x} +  \int_{{\partial \Omega_N(t)}} g_N v \mathrm{~d}S.
\end{aligned}
\end{equation}

\par A detailed discussion on well-posedness of (\ref{eq:global_weak_problem}) can be found in \cite[Sec.~10.3]{Gross2011} for a mass transport problem with a moving interface and an advection velocity $\boldsymbol{w}$. Well-posedness requires that $\boldsymbol{w}\cdot\boldsymbol{n_x}+n_t {\geq} 0$ on the Neumann boundary $\partial Q_N$, where $\boldsymbol{n} = (\boldsymbol{n_x}, n_t)$ denotes the outward normal to $\partial Q$.

\subsection{Discrete formulation}\label{sec:discret_formulation}

\par In order to state the discrete problem, we introduce some notation. Let $\phi : Q \rightarrow \mathbb{R}$. We denote the values of the function at both sides of an inter-slab interface $t^n$ as $\phi^{n,\pm} \doteq \lim_{\epsilon \rightarrow 0} \phi(\boldsymbol{x},t^n \pm \epsilon)$ and its \emph{jump} as $\llbracket \phi \rrbracket^{n} \doteq \phi^{n,+} - \phi^{n,-}$. Also, we set $\phi^{0,-} \doteq \phi(\boldsymbol{x},t^0)$. 

We consider a weak imposition of the Dirichlet boundary conditions using the Nitsche's method \cite{Nitsche1971} and a spatial \ac{cg} approximation. In any case, the extension to \ac{dg} in space is straightforward, since \ac{dg} methods can readily be applied to agglomerated meshes. 

Since the coupling between time slabs respect causality, it is sufficient to study the problem on a single time slab assuming the value of the unknown at the previous one is known. 
The weak formulation of the model problem \eqref{eq:heat_eqn} using a spatial \ac{cg} and a temporal \ac{dg} discretisation on a time slab $J^n$ consists of finding 
\begin{equation}
	\label{eq:weak_form_st}
	u_h \in V_{h,ag}^n \quad : \quad B_h^n(u_h,v_h) = L_h^n(v_h) \quad \forall  v_h \in V_{h,ag}^n,
\end{equation}
where the left-hand side reads
\begin{equation}\label{eq:FEOp_slabwise_lhs}
	\begin{aligned}
		B_h^n(u_h,v_h) =& \int_{Q^n} \partial_t u_h \mathrm{~} v_h \mathrm{~d}\boldsymbol{x} \mathrm{~d}t 
		+ \int_{\Omega^{n-1}} \llbracket u_h \rrbracket^{n-1} {v}_{h}^{n-1,+} \mathrm{~d}\boldsymbol{x} \\ &+ \int_{J^n} a_h(t,u_h,v_h) \mathrm{~d} t + \int_{J^n} c(t,w,u_h,v_h) \mathrm{~d} t,\\
		a_h(t,u_h,v_h) =&   \int_{\Omega(t)} \mu \nabla u_h \cdot \nabla v_h \mathrm{~d}\boldsymbol{x} 
		+ \sum_{T \in \mathcal{T}_{h,act}^n} \int_{{\partial \Omega_D(t)} \cap \bar{T}} \beta_{\bar{T}} u_h v_h \mathrm{~d}S  \\
		&\quad -\int_{{\partial \Omega_D(t)} } \mu (\boldsymbol{n_x}\cdot \nabla u_h)v_h + \mu (\boldsymbol{n_x}\cdot \nabla v_h)u_h \mathrm{~d}S,
  \end{aligned}
\end{equation}
and the right-hand side reads
  \begin{equation}\label{eq:FEOp_slabwise_rhs}
	\begin{aligned}
		L^n_h(v_h) &= \int_{J^n} l_h(t,v_h) \mathrm{~d}t, \\
		l_h(t,v_h) &= \int_{\Omega(t) } f \mathrm{~} v_h \mathrm{~d}\boldsymbol{x} + 
    \sum_{T \in \mathcal{T}_{h,act}^n} \int_{{\partial \Omega_D(t)} \cap \bar{T}} \beta_{\bar{T}} g_D v_h - \mu (\boldsymbol{n_x}\cdot \nabla v_h)g_D \mathrm{~d}S  \\ &\quad + \int_{{\partial\Omega_N(t)} } g_N v_h \mathrm{~d}S.
	\end{aligned}
\end{equation}
$\boldsymbol{n_x}$ is the space-only normal vector, $\beta_{\bar{T}} = \mu \gamma /{h_{\bar{T}}}$ and $\gamma >0$ is the Nitsche's parameter, which is independent of the cut-cell configuration and must be \emph{large enough} for stability purposes. The value $u_h^{n-1,-}$ comes from the solution of the previous time slab or the initial condition, where we make use of the causality in time. The global \ac{fe} problem over the whole time domain reads as: find 
\begin{equation}\label{eq:whole-problem}
  u_h \in \mathcal{V}_{h,ag} \ : \  B_h(u_h,v_h) = L_h(v_h), \qquad \forall v_h \in \mathcal{V}_{h,ag},
\end{equation}  
with
\begin{equation*}
	\begin{aligned}
	B_h(u_h,v_h) & = \sum \limits_{n=1}^N B^n_h(u_h^n,v_h^n) + \int_{\Omega^0}^{}  u_h^{0,-} v_h^0 ~\mathrm{d}\boldsymbol{x} , \qquad 
	L_h(v_h)&= \sum \limits_{n=1}^N L^n_h(v_h^n) +  \int_{\Omega^0} u^{0} v_{h}^{0} \mathrm{~d}\boldsymbol{x}.
\end{aligned}
\end{equation*}

\subsection{Ghost penalty methods}\label{sec:ghost_penalty}
  In this work, we restrict ourselves to aggregation (or agglomeration) approaches for the sake of conciseness. However, the ideas and numerical analyses in this work could be extended to ghost penalty stabilisation techniques. The proposed space-time ghost penalty method reads: find
\begin{equation}\label{eq:whole-problem-ghost}
  u_h \in \mathcal{V}_{h,act} \ : \  B_h(u_h,v_h) + \sum_{n=1}^{N} \int_{J^n} s_h(u_h,v_h) = L_h(v_h), \qquad \forall v_h \in \mathcal{V}_{h,act},
\end{equation}  
  where at each time slab 
  \begin{equation}
   s_h(u_h,v_h)  = \sum_{\bar{T} \in \bar{\mathcal{T}}_{h,cut}^n}^{} \int_{\bar{T}}   \gamma h_{\bar{T}}^{-2}(u_h - \Pi_{gh}(u_h))(v_h - \Pi_{gh}(v_h))\mathrm{d} \boldsymbol{x},
  \end{equation}
  in which $\Pi_{gh}(\cdot)$ is an operator that performs a projection in space only. For instance, we can simply take $\Pi_{gh}(\cdot) = \mathcal{E}^n$ at each time slab. This way, we obtain a space-time version of the weak \ac{agfem} proposed in \cite{Badia2022ghost}. Alternatively, we can use in space a standard projection onto $\mathcal{P}_p(U)$, where $U$ is an aggregate to define a space-time extension of the standard bulk-based ghost penalty method (see \cite{Badia2022ghost,Burman2010ghost} for more details). Alternatively, we can consider a face-based ghost penalty stabilisation
  \begin{equation}
   s_h(u_h,v_h)  = \sum_{{F} \in {\mathcal{{F}}}_{h,cut}^n}^{} \int_{\bar{F}}   \gamma h_F
   \jump{{\partial_{\boldsymbol{n}_{\boldsymbol{x}}}^i  u_h}} \cdot  
   \jump{{\partial_{\boldsymbol{n}_{\boldsymbol{x}}}^i  v_h}} \mathrm{d} \boldsymbol{x},
  \end{equation}
  where we extend the notation for cells to faces {and $\partial_{\boldsymbol{n}_{\boldsymbol{x}}}^i$ denotes the normal derivative of order $i$}. ($\mathcal{F}_{h,cut}^{n}$ are the cut faces in the space-time mesh $\mathcal{T}_{h}^n$.  Given a face $F \in \mathcal{F}_{h,cut}^{n}$, we represent with $\bar{F}$ the space-only face such that $F = \bar{F} \otimes J^n$.)
 
  The reason for the choices above (analogously to the \ac{agfem} case) is that we can ensure the desired extended stability, continuity and weak consistency properties in \cite[Def. 4.1]{Badia2022ghost} at all time values.

\section{Numerical analysis}\label{sec:numerical_analysis}

\par In this section, we analyse the numerical properties of the space-time \ac{agfem} proposed above. The analysis for the ghost penalty method in Sec.~\ref*{sec:ghost_penalty} can be performed in an analogous way and is not considered for conciseness. We provide first the well-posedness of the steady problem. Next, we consider the transient problem and space-time discretisation. Our time discretisation makes use of \ac{dg} methods and the space-only discretisation can vary between slabs (due to the cell-wise aggregation and possibly space refinement). Following similar ideas in \cite{Chrysafinos2006error} (for body-fitted formulations), we consider an $L^2(Q)$ projector in space-time, in order to eliminate the time derivative terms in the \emph{a priori} error analysis. Anisotropic error estimates are obtained for this projector, relying on an extension of the solution to $Q_{act}\doteq \cup_{n=1}^N Q^n_{act}$. Finally, for the space-only terms of the bilinear form, we can readily use previous analyses of \ac{agfem} in the steady case (see, e.g., \cite{Badia2018aggregated,Badia2022robust}).

For simplicity, we assume $\boldsymbol{w} = \boldsymbol{0}$ in the analysis. As a result, we require $n_t \geq 0$ for well-posedness of the problem. However, the analysis can readily be extended to more complex models, e.g., convection-diffusion-reaction systems, possibly with numerical stabilisation. A robust analysis for singularly perturbed limits (i.e., high Peclet and Reynolds numbers) and error bounds at arbitrary time values can be carried out using the technique proposed in \cite[Section 3]{Chrysafinos2006error} in the analysis below. We also consider an exact treatment of the geometry. Integration errors can not be warded off for non-polyhedral domains and the discrete formulation requires a careful geometrical treatment (see, e.g., \cite{preuss,heimann_ma} for high order space-time geometrical approximations).

\par In the following analysis, all constants are independent of mesh size and the location of cut cells but can depend on the polynomial order. We also introduce the following notation, if $A \le cB$, where $c$ is a positive constant, then we write $A \lesssim B$; similarly if $A \ge cB$, then $A \gtrsim B$. 

\subsection{Spatial discretisation}\label{sec:spatial_disc}

\par In this section, we prove the well-posedness of the spatial discretisation. In the following sections, we will make use of these results at fixed time values $t \in (0,T]$. To avoid cumbersome notation, we drop the time dependency; $\Omega(t)$ and the restriction of space-time \ac{fe} spaces and meshes at $t$ are represented with $\Omega$, $\bar{V}_h$ and $\bar{\mathcal{T}}_h$, resp. We also drop the time slab superscript from \ac{fe} spaces and meshes.  

The spatial discretisation corresponding to the model problem is to seek 
\[
  u_h \in \bar{V}_{h,ag}  \, : \, a_h(u_h,v_h) = l_h(v_h), \quad \forall v_h \in \bar{V}_{h,ag}.
\]
We define the norm on $\bar{V}_{h,ag}$ as
\[
\|v\|^2_{\bar{V}_h} \doteq \mu \|\nabla v \|^2_{L^2(\Omega)} + \sum_{\bar{T} \in \bar{\mathcal{T}}_{h,act}} \beta_{\bar{T}} \|v\|^2_{L^2(\bar{T} \cap \partial\Omega_{D})}.
\]
We introduce the space $\bar{V}(h) \doteq \bar{V}_{h,ag} + H^2(\Omega)$ endowed with the norm
\[
\|v\|_{\bar{V}(h)}^2 \doteq \mu \|\nabla v \|^2_{L^2(\Omega)} + \sum_{\bar{T} \in \bar{\mathcal{T}}_{h,act}}\beta_{\bar{T}} \|v\|^2_{L^2(\bar{T} \cap \partial \Omega_{D})} + \sum_{\bar{T} \in \bar{\mathcal{T}}_{h,act}} \mu h_{\bar{T}}^2 |v|^2_{H^2(\bar{T} \cap \Omega)}.
\]
Using a discrete inverse inequality for \ac{fe} functions on \ac{agfe} spaces \cite{Neiva2021}, it can be proved that $\|v_h\|_{\bar{V}(h)} \lesssim \|v_h\|_{\bar{V}_h}$, for any $v_h \in \bar{V}(h)$. {We also introduce a trace inequality to estimate the Nitsche terms. For a domain $\omega$ with a Lipschitz boundary, the following inequality holds (see, \mbox{\cite[Th. 1.6.6]{Brenner2008}}):
\begin{equation}\label{eq:trace_ineq_agg}
	\|u\|^2_{L^2(\partial \omega)} \le C_\omega \|u\|_{L^2(\omega)}\|u\|_{H^1(\omega)}, \hspace{2ex} u \in H^1(\omega).
\end{equation}
The constant $C_\omega$ depends only on the shape of $\omega$. Since the aggregates are shape regular, this inequality holds at the aggregate level also. We also make use of an inverse inequality for aggregates. Let $A$ denote an aggregate. For any $u_h \in \bar{V}_{h,ag}$ and any aggregate $A \in \mathrm{ag}(\bar{\mathcal{T}}_{h,act})$, we have (see \cite{Badia2022robust})
\begin{equation}\label{eq:inverse_ineq_agg}
	\|u_h\|_{H^1(\Omega\cap A)} \le C h_A^{-1}\|u_h\|_{L^2(\Omega \cap A)},
\end{equation}    
where, $C>0$ is a constant and $h_A$ is the size of the aggregate.
}

\par The coercivity and continuity of the bilinear form ensure the well-posedness of the problem.
\begin{proposition}\label{prop:coercivity_space}
	The bilinear form $a_h(\cdot,\cdot)$ satisfies:
  \begin{equation}\label{eq:well_posed_space}
    a_h(v_h,v_h) \ge c_\mu  \left( \|v_h\|^2_{\bar{V}_h} + \|v_h\|^2_{L^2(\Omega)} \right), \quad a_h(u,v_h) \le C_\mu \|u\|_{\bar{V}(h)} \|v_h\|_{\bar{V}(h)},
  \end{equation}
  for any $v_h \in \bar{V}_{h,ag}$ and $u \in \bar{V}(h)$ and  
 $\gamma$ large enough, and where $c_\mu$ and $C_\mu$ are positive constants away from zero.
\end{proposition}
\begin{proof}
  The proof of continuity and coercivity with respect to the norm $\|\cdot\|_{\bar{V}(h)}$ can be found in \cite{Badia2022robust}.   
  We control the $L^2$ term using a generalised Poincar\'e inequality. We define
	\[
		f(u) =|\partial \Omega_D|^{-1/2}\int_{\partial \Omega_D} u \mathrm{~d}S.
	\]
	The restriction of $f$ on constant functions is non-zero. As $\bar{V}_{h,ag} \subset H^1(\Omega)$, using \cite[Lemma.~B.63]{Ern2004} and Cauchy Schwarz inequality yields
	\begin{align*}
		c_{p,\Omega}\|v_h\|_{L^2(\Omega)} &\le\|\nabla v_h \|_{L^2(\Omega)} + |f(v_h)| \lesssim \|\nabla v_h \|_{L^2(\Omega)} + \|v_h\|_{L^2(\partial \Omega_D)}, \label{eq:extended_poincare}
	\end{align*}
	where $c_{p,\Omega} >0$ is a constant that depends only on the domain and order of spatial discretisation.
\end{proof}

\subsection{Stability analysis}\label{sec:stab_st}

\par In this section, we analyse the stability of the fully discretised space-time problem. We introduce a \ac{dg} norm on $V_{h,ag}^n$ as 
\begin{equation}\label{eq:norm_slab}
	\tnor{v}_n^2 \doteq \|v^{n,-}\|^2_{L^2(\Omega^{n})} + \|v^{n-1,+}-v^{n-1,-} \|^2_{L^2(\Omega^{{n-1}})} + c_\mu \int_{J^{n}}^{} \|{v}\|^2_{\bar{V}(h)} \mathrm{~d}t,
\end{equation}
and an accumulated \ac{dg} norm on $V_{h,ag}^1 \times \dots \times V_{h,ag}^n$ as
\begin{equation}\label{eq:norm_st}
	\tnor{v}_{n,*}^2 \doteq  \|v^{n,-}\|^2_{L^2(\Omega^{n})} 
	+\sum_{i=0}^{n-1} \|v^{i,+}-v^{i,-} \|^2_{L^2(\Omega^{{i}})} 
	+ c_\mu \int_{0}^{t^n} \|{v}\|^2_{\bar{V}(h)} \mathrm{~d}t.
\end{equation}
We make repeated use of the following property on the space-time domain. We represent with $n_t$ the temporal component of the space-time normal vector $\boldsymbol{n}$ on $\partial Q$. 
\begin{proposition}\label{prop:time_dep_terms_equality}
	Any function $u \in H^{(0,1)}(Q^n)$ satisfies 
	\begin{equation}
		\begin{aligned}\label{eq:time_terms_equality}
				\int_{Q^n} \partial_t u \mathrm{~}u \mathrm{~d}\boldsymbol{x} \mathrm{~d}t +\int_{\Omega^{n-1}} \llbracket u \rrbracket^{n-1} {u}^{n-1,+} \mathrm{~d}\boldsymbol{x} -  \frac{1}{2} \int_{\partial Q^n_D \cup \partial Q^n_N} n_t u^2  \mathrm{~d}S \\
			= \frac{1}{2}\| u^{n,-} \|_{L^2(\Omega^n)}^2 - \frac{1}{2}\| u^{n-1,-} \|_{L^2(\Omega^{n-1})}^2 + \frac{1}{2}\| u^{n-1,+} - u^{n-1,-} \|_{L^2(\Omega^{n-1})}^2  .
		\end{aligned}
	\end{equation} 
\end{proposition}

\begin{proof}
	The boundaries of $Q^n$ are $\Omega^n, \Omega^{n-1},\partial Q^n_D,\partial Q^n_N$. 
  The temporal component $n_t$ of the space-time normal vector $\boldsymbol{n}$ takes the following values on $\partial Q^n$ : $n_t = -1$ on $\Omega^{n-1}$ and $1$ on $\Omega^n$.    
	Using the Gauss-Green theorem, 
  we get
	\begin{equation*}
		\begin{aligned}
			\int_{Q^n} \partial_t u \mathrm{~}u \mathrm{~d}\boldsymbol{x} \mathrm{~d}t  &= \frac{1}{2} \int_{\partial Q^n} n_t u^2  \mathrm{~d}S \\&= \frac{1}{2} \|u^{n,-}\|_{L^2(\Omega^n)}^2 -  \frac{1}{2} \|u^{n-1,+}\|_{L^2(\Omega^{n-1})}^2 + \frac{1}{2} \int_{\partial Q^n_D \cup \partial Q^n_N} n_t u^2  \mathrm{~d}S. 
		\end{aligned}
	\end{equation*}
After some algebraic manipulations, we get:
	\begin{align*}
		\int_{\Omega^{n-1}} \llbracket u \rrbracket^{n-1} {u}^{n-1,+} \mathrm{~d}\boldsymbol{x} = \frac{1}{2} \|u^{n-1,+}\|_{L^2(\Omega^{n-1})}^2 + \frac{1}{2} \|u^{n-1,+} -u^{n-1,-} \|_{L^2(\Omega^{n-1})}^2 - \frac{1}{2} \|u^{n-1,-}\|_{L^2(\Omega^{n-1})}^2.
	\end{align*}
  Combining these results, we prove the proposition.
\end{proof}

\begin{proposition}[Local stability estimate]\label{prop:stability_st}
	The bilinear form $B_h^n(\cdot,\cdot)$ satisfies 
  \[
   c_\mu \| v_h \|_{L^2(Q^n)}^2  +  \tnor{v_h}_n^2 \lesssim B_h^n(v_h,v_h) + \|v_{h}^{n-1,-}\|_{L^2(\Omega^{n-1})}^2, \ \forall \ v_h \in {V}_{h,ag}^n
  \] 
  for $\gamma$ large enough.
\end{proposition}
\begin{proof}
	Let $v_h \in V^n_{h,ag}$. Using the coercivity of the spatial bilinear form for all $t \in (0,T]$ in Prop.~\ref{prop:coercivity_space} and  \eqref{eq:time_terms_equality}, we get:
	\begin{align*}
		& B_h^n(v_h,v_h) = \int_{Q^n} \partial_t v_h \mathrm{~} v_h \mathrm{~d}\boldsymbol{x} \mathrm{~d}t 
		+ \int_{\Omega^{n-1}} \llbracket v_h \rrbracket^{n-1} {v}_{h}^{n-1,+} \mathrm{~d}\boldsymbol{x} + \int_{J^n} a_h(t,v_h,v_h) \mathrm{~d} t \\
    & \geq   \frac{1}{2}\| v_h^{n,-} \|_{L^2(\Omega^n)}^2   + \frac{1}{2}\| v_h^{n-1,+} - v_h^{n-1,-} \|_{L^2(\Omega^{n-1})}^2 + c_\mu \int_{J^n}^{}   \left( \|v_h\|^2_{\bar{V}_h} + \|v_h\|^2_{L^2(\Omega(t))} \right)\mathrm{d} t \\
    & \qquad  - \frac{1}{2}\| v_h^{n-1,-} \|_{L^2(\Omega^{n-1})}^2 + \frac{1}{2} \int_{\partial Q^n_D \cup \partial Q^n_N} n_t v_h^2  \mathrm{~d}S.
	\end{align*} 
	We require $n_t >0 \textrm{ on } \partial Q_N^n$ for the continuous problem to be well posed. Since $|n_t| \le 1$, choosing $\gamma$ such that {$\beta_{\bar{T}} > 1/{c_\mu}$}, we get
	\begin{equation*}
		\tnor{v_h}_n^2 +  c_\mu \| v_h \|_{L^2(Q^n)}^2 \lesssim  \| v^{n-1,-}_{h} \|_{L^2(\Omega^{n-1})}^2 + B_h^n(v_h,v_h).
	\end{equation*}
\end{proof}

\begin{proposition}[Continuity]\label{prop:continuity_st}
	The bilinear form $B_h^n(\cdot,\cdot)$ satisfies
  \begin{align}\label{eq:continuity}
        B_{h}^{n}(u,v_{h}) \lesssim 
        & \left( \|\partial_t u\|_{L^2(Q^n)}^2 + \|u^{n-1,+}-u^{n-1,-}\|^2_{L^2(\Omega^{n-1})} + C_\mu \int_{J^n}^{} \|u\|^2_{\bar{V}(h)}\mathrm{~d}t\right) \nonumber\\ & \times 
        \left( \| v_h\|_{L^2(Q^n)}^2 + \|v^{n-1,+}_{h}\|^2_{L^2(\Omega^{n-1})} + C_\mu \int_{J^{n}}^{} \|v_h\|^2_{\bar{V}(h)}\mathrm{~d}t\right)
  \end{align}
  for any $u \in H^{(2,1)}(Q^n)$ and $v_h \in V_{h,ag}^{n}$. 
\end{proposition}
\begin{proof}
The result can readily be obtained using Cauchy-Schwarz inequality and the continuity in Prop.~\ref{prop:coercivity_space}.
\end{proof}

\begin{proposition}[Galerkin orthogonality]\label{prop:orthogonality}
If $u$ solves the model problem \eqref{eq:heat_eqn} and $u_h$ solves the discrete problem \eqref{eq:weak_form_st}, then $B_h^n(u-u_h,v_h) = 0$ for all $u \in H^{(2,1)}(Q)$ and $v_h \in V_{h,ag}^n$.
\end{proposition}
\begin{proof}
	Since $u \in H^{(2,1)}(Q), \llbracket u \rrbracket^n = 0$ a.e. in $\Omega^n$, for $n = 0,1,2, \dots N$. Using $u = g_D$ on $\partial Q_D^n$ and integration by parts yields $B_h^n(u,v_h) = L_h^n(v_h)$, for all $v_h \in V_{h,ag}^n \subset H^{(1,0)}(Q^n)$. Using \eqref{eq:weak_form_st}, we prove the result.
\end{proof}

\subsection{Anisotropic space-time approximation}\label{sub:interpolators}

\par In this section, we introduce a space-time $L^2$ projector $\pi_{h\tau}^n$  that will be useful in the \emph{a priori} error analysis. The key properties of this projector are that it satisfies anisotropic error estimates and the corresponding approximation error cancels the time derivative term.

To do so, we define several projectors, some of them defined on functions extended to the active domain $Q^n_{act}$. This extension is needed in the following analysis to end up with anisotropic estimates for the space-time \ac{agfe} spaces. Isotropic estimates can be readily obtained without this requirement by using the interpolator in \cite[Lemma~5.11]{Badia2018aggregated} in the space-time domain and the continuity of $\mathcal{E}^n$. The continuity of $\mathcal{E}^n$ is a consequence of its tensor-product definition and the continuity of the space-only extension in \cite[Cor.~5.3]{Badia2018aggregated}: 
\begin{equation}\label{eq:continuity_extension}
	\|\bar{\mathcal{E}} v\|_{L^2(\Omega_{act})} \lesssim \|v\|_{L^2(\Omega_{in})} , \quad \forall v \in \bar{V}_{h,in}.
\end{equation}
First, we introduce $L^2$ projectors for space-only (in $\Omega_{act}^n$ and $\Omega$) and time-only functions (in $J^n$), resp.:
\begin{align*} &
\bar{\pi}^n_{h,act}: L^{2}(\Omega_{act}^n) \rightarrow \bar{V}_{h,ag}^n \ : \ \int_{\Omega^n_{act}} (u - \bar{\pi}_{h,act}^n(u)) v_h \mathrm{~d} \boldsymbol{x} = 0, \quad \forall v_h \in \bar{V}_{h,ag}^n, \\ &
  \bar{\pi}^n_{h}: L^{2}(\Omega^{n}) \rightarrow \bar{V}_{h,ag}^n \ : \
\int_{\Omega^n} (u - \bar{\pi}_{h}^n(u)) v_h \mathrm{~d} \boldsymbol{x} = 0, \quad \forall v_h \in \bar{V}_{h,ag}^n, \\ &
\bar{\pi}_{\tau}^{n}: L^2(J^n) \rightarrow \mathcal{P}_{q}(J^n) \ : \ \bar{\pi}^n_{\tau}(u)(t^n) = u(t^n)
\ , \ \int_{J^n} (u - \bar{\pi}_{\tau}^n(u)) v_\tau \mathrm{~d} t = 0, \quad \forall v_\tau \in \mathcal{P}_{q-1}(J^n),
\end{align*} 
for $q>0$. Next, we define space-time semi-discrete projectors:
\begin{align*} &
\pi^n_{h,act}: L^2(Q^n_{act}) \rightarrow \bar{V}_{h,ag}^n \otimes L^2(J^n) \ : \ \int_{Q^n_{act}} (u - \pi_{h,act}^n(u)) v_h \mathrm{~d} \boldsymbol{x} \mathrm{~d} t= 0, \quad \forall v_h \in \bar{V}_{h,ag}^n \otimes L^2(J^n), \\ &
  {\pi}_{\tau,act}^n: L^2(\Omega_{act}^n) \otimes \mathcal{C}^0(J^n) \rightarrow L^2(\Omega_{act}^n) \otimes \mathcal{P}_{q}(J^n) \ : \ \pi_{\tau,act}^n(u)(t^n) = u(t^n), \\ & \qquad \qquad \qquad \int_{Q^n_{act}}^{} (u - \pi_{\tau,act}^n(u)) v_\tau  \mathrm{~d} \boldsymbol{x} \mathrm{~d} t = 0, \quad \forall v_\tau \in L^2(\Omega^n_{act}) \otimes \mathcal{P}_{q-1}(J^n), \quad q > 0.
\end{align*}  

\begin{lemma}\label{lm:equivalence-projections}
  The following equivalences hold in $L^2(Q^n_{act})$-sense:
  \begin{align}
    & \bar{\pi}_{h,act}^{n}(u(\cdot,t))(\boldsymbol{x})  = \pi_{h,act}^n (u)(\boldsymbol{x},t), \quad && \forall u \in L^2(\Omega^n_{act}) \otimes \mathcal{C}^0(J^n),\label{eq:equiv-space}\\ & 
    \bar{\pi}_{\tau}^{n}(u(\boldsymbol{x},\cdot))(t)  = \pi_{\tau,act}^n (u)(\boldsymbol{x},t),\quad && \forall u \in H^1(\Omega^n_{act}) \otimes L^2(J^n).\label{eq:equiv-time}
  \end{align}   
\end{lemma}
\begin{proof}
The first result has been proven in \cite[Lemma~3.3.5]{Lehrenfeld:462743}. The second result can be proved using the fact that one can approximate any function in $u \in H^1(\Omega^n_{act})$ by $u_\epsilon \in \mathcal{C}^0(\Omega^n_{act})$ such that $\|u - u_\epsilon\|_{L^2(\Omega^n_{act})} \to 0$ as $\epsilon \to 0$. 
\end{proof}

Finally, we define the fully discrete space-time projector \mbox{$\pi_{h\tau}^n : H^{(0,1)}(Q^n) \rightarrow V_{h,ag}^n = \bar{V}_{h,ag}^n \otimes \mathcal{P}_{q}(J^n)$} such that $\pi_{h\tau}^n (u)(t^n)= \bar{\pi}_{h}^n(u(t^n))$ and
		\begin{equation}\label{eq:local_projector_fem}
			\int_{Q^n} (u-\pi^n_{h\tau} u) v_{h\tau} \mathrm{~d} \boldsymbol{x} \mathrm{~d}t = 0, \mathrm{~} \quad \forall v_{h\tau} \in \bar{V}_{h,ag}^{n}\otimes \mathcal{P}_{q-1}(J^n) , \quad q > 0, 
		\end{equation} 
and the analogous on $Q^n_{act}$, represented with  $\pi_{h\tau,act}^{n}$ and defined as the one above by replacing $Q^n \leftarrow Q^n_{act}$. 

Next, we prove approximation error bounds for the space and space-time \ac{agfe} spaces. In the following result, we denote with $\Omega_{act}$ the set of spatial active cells at any given time $t \in(0, T]$. The time dependency is eliminated to avoid cumbersome notations.

\begin{proposition}[Approximation error in space]\label{prop:proj_error_time_points}
	Let $u \in H^{p+1}(\Omega_{act})$ and $\bar{V}_{h,ag}$ be of order $p$. The following result holds:
	\begin{equation}\label{eq:proj_err_time_point}
		\| u - \bar{\pi}_{h,act}(u) \|_{H^s(\Omega)} \lesssim  h^{p+1-s} |u |_{H^{p+1}(\Omega_{act})}, \quad \hbox{for $0 \leq s \leq p+1$.}
	\end{equation}
\end{proposition}
\begin{proof}
  This bound can readily be proved using the existence of an optimal interpolant $\bar{\mathcal{I}}_{h}$ onto $\bar{V}_{h,ag}$ (see \cite[Lemma~5.11]{Badia2018aggregated}), a standard inverse inequality in space, and the stability of the $L^2$ projection as follows:   
  \begin{align}
         \|u - \bar{\pi}_{h,act}(u) \|_{H^s(\Omega)} & \leq 
         \|u -  \bar{\mathcal{I}}_{h}(u) \|_{H^s(\Omega)} + 
         \|  \bar{\pi}_{h,act}(u)-\bar{\mathcal{I}}_{h}(u) \|_{H^s(\Omega)} \label{eq:typical-approx}\\ & \lesssim 
         \|u -  \bar{\mathcal{I}}_{h}(u) \|_{H^s(\Omega)} +
         h^{-s}\| \bar{\pi}_{h,act} (u-\bar{\mathcal{I}}_{h}(u)) \|_{L^2(\Omega_{act})} \nonumber \\ & \lesssim
         \|u - \bar{\mathcal{I}}_{h}(u) \|_{H^s(\Omega)} +
         h^{-s}\|u-\bar{\mathcal{I}}_{h}(u) \|_{L^2(\Omega_{act})} \nonumber \\ & \lesssim
         h^{p+1-s} |u |_{H^{p+1}(\Omega_{act})}. \nonumber
  \end{align}
\end{proof}

  \begin{proposition}[Approximation error in space-time]\label{prop:proj_error_st}
    Let $u \in H^{(p+1,q+1)}(Q^n_{act})$ and \linebreak $V_{h,ag}^n \doteq \bar{V}_{h,ag}^n \otimes \mathcal{P}_{q}(J^n)$, with $\bar{V}_{h,ag}^n$ of order $p$ and $q>0$. The following results hold for $0 \le s \le p+1$:
    \begin{align*}
      & \|u - {\pi}_{h,act}^n(u) \|_{H^{(s,0)}(Q^n)}  \lesssim 
      h^{p+1-s} |u |_{H^{(p+1,0)}(Q^n_{act})}, \\ & 
      \|u - \pi_{\tau,act}^n(u)\|_{L^2(Q^n)}  \lesssim \tau^{q+1}|u|_{H^{(0,q+1)}(Q_{act}^n)},\\ & 
      \|u - \pi_{h\tau,act}^n(u)\|_{H^{(s,0)}(Q^n)}  \lesssim h^{p+1-s} |u |_{H^{(p+1,0)}(Q^n_{act})} + h^{-s}\tau^{q+1}|u|_{H^{(0,q+1)}(Q_{act}^n)}.
    \end{align*}
  \end{proposition}
\begin{proof}
Since $u$ is continuous in time, using the optimal interpolant $\bar{\mathcal{I}}_{h}$  at each time value and its approximability properties, a standard inverse inequality in space and the stability of $\pi_{h,act}^n$ in $L^2(Q^n_{act})$, we obtain:      
  \begin{align*}
    \int_{J^n} \|u - {\pi}_{h,act}^n(u) \|_{H^s(\Omega)}^2 & \leq 
    \int_{J^n} \|u -  \bar{\mathcal{I}}_{h}(u) \|_{H^s(\Omega)}^2 + 
    \| {\pi}_{h,act}^n(u- \bar{\mathcal{I}}_{h}(u)) \|_{H^s(\Omega)}^2 \\ & \lesssim \int_{J^n}
    \|u -  \bar{\mathcal{I}}_{h}(u) \|_{H^s(\Omega)}^2 +
    h^{-2s}\| {\pi}_{h,act}^n(u-\bar{\mathcal{I}}_{h}(u)) \|_{L^2(\Omega_{act})}^2 \\ & \lesssim
    \int_{J^n} \|u - \bar{\mathcal{I}}_{h}(u) \|_{H^s(\Omega)}^2 +
    h^{-2s}\|u-\bar{\mathcal{I}}_{h}(u) \|^2_{L^2(\Omega_{act})} \\ & \leq
    h^{2(p+1-s)} |u |^2_{H^{(p+1,0)}(Q^n_{act})}.
\end{align*}
The time approximation error can be proved using the equivalence in (\ref{eq:equiv-time}) and the approximation properties of $\bar{\pi}_{\tau}^n$ (see \cite[Th.~12.1]{Thome}):
\begin{align*}
\|u - \pi_{\tau,act}^n(u)\|_{L^2(Q^n)}^2 & \leq \int_{\Omega_{act}} \|u - \bar{\pi}_{\tau}^n(u)\|_{L^2(J^n)}^2 \\ & \lesssim \int_{\Omega_{act}} \tau^{2(q+1)}|u|_{H^{q+1}(J^n)}^2 = \tau^{2(q+1)}|u|_{H^{(0,q+1)}(Q^n_{act})}^2.
\end{align*}
The last result can be obtained combining the two previous results, the continuity of $u$ in time, the stability of $\pi_{h,act}^n$, and an inverse inequality in space:
\begin{align*}
& \|u - \pi_{h\tau,act}^n(u)\|_{H^{(s,0)}(Q^n)} \leq \|u - \pi_{h,act}^n(u)\|_{H^{(s,0)}(Q^n)} + \|\pi_{h,act}^n(u) - \pi_{h\tau,act}^n(u)\|_{H^{(s,0)}(Q^n)} \\ & \lesssim \|u - \pi_{h,act}^n(u)\|_{H^{(s,0)}(Q^n)} + h^{-s} \|\pi_{h,act}^n(u - \pi_{\tau,act}^n(u))\|_{L^{2}(Q^n_{act})} \\ & \lesssim  \|u - \pi_{h,act}^n(u)\|_{H^{(s,0)}(Q^n)} + h^{-s} \|u - \pi_{\tau,act}^n(u)\|_{L^{2}(Q^n_{act})}.
\end{align*}
\end{proof}

\begin{proposition}[Approximation error in $\int_{J^n}\|\cdot\|_{\bar{V}(h)}$.]\label{prop:proj_error_vh}
	If $u \in H^{(p+1,q+1)}(Q^n_{act})$ and ${V}_{h,ag}^n$ has order $p$ in space and $q>0$ in time, then
	\begin{align}\label{eq:bound_vh_norm}
		\int_{J^n} \|u -\pi^n_{h\tau}(u) \|^2_{\bar{V}(h)} \mathrm{~d} t \lesssim  h^{2p}|u|^2_{H^{(p+1,0)}(Q^n_{act})} + h^{-2}\tau^{2(q+1)}|u|_{H^{(0,q+1)}(Q^n_{act})}^2.
	\end{align}
\end{proposition}

\begin{proof}
  Let us consider the bound for $H^1(\Omega)$ term  in $\bar{V}(h)$. We note that $\pi_{h\tau}^n \pi_{h\tau,act}^n =   \pi_{h\tau,act}^n$ by construction (both are projections onto the same discrete space). Using this fact together with the inverse inequality in space and the stability of $\pi_{h\tau}^n$ in $L^2(Q^n)$, we obtain:
  \begin{align*}
    \|u - \pi_{h\tau}^n (u) \|_{H^{(1,0)}(Q^n)}^2 & \leq 
      \|u -  \pi_{h\tau,act}^n (u) \|_{H^{(1,0)}(Q^n)}^2 + \|\pi_{h\tau}^n(u) - \pi_{h\tau,act}^n (u) \|_{H^{(1,0)}(Q^n)}^2 \\ & \lesssim \|u -  \pi_{h\tau,act}^n (u) \|_{H^{(1,0)}(Q^n)}^2 + h^{-2} \| u - \pi_{h\tau,act}^n (u) \|_{L^2(Q^n)}^2. 
  \end{align*}
  The other terms in the $\bar{V}(h)$ norm can readily be obtained using the same ingredients.
\end{proof}

\subsection{Error estimates}\label{sec:error_estimates}
\par Let $u$ be the solution of ~\eqref{eq:heat_eqn} and $u_h$ be the solution of ~\eqref{eq:whole-problem}. We can express the total error as the sum of the approximation error $e_h \doteq \pi_{h\tau}u - u_h$ and the projection error $e_p \doteq (u-\pi_{h\tau}u)$. The total error is defined as 
\begin{equation}\label{eq:error_split}
	e = u - u_h = (u-\pi_{h\tau}u) + (\pi_{h\tau}u - u_h) = e_p + e_h.
\end{equation}

\begin{proposition}[Bounds for approximation error]\label{prop:approx_error_bounds}
  The following inequality holds:
		\begin{align*}
			\tnor{e_h}_{N,*}^2 &\lesssim \|e_{h}^{0,-}\|^2_{L^2(\Omega^{0})} + \int_{0}^{T} \|u-\pi_{h\tau} u\|_{\bar{V}(h)}^2 \mathrm{~d}t 
      + \sum \limits_{n=0}^{N-1}
        \|(I-\bar{\pi}_{h}^n)u(t^n)^{-}\|^2_{L^2(\Omega^{n})}.
		\end{align*}
\end{proposition}

\begin{proof}
  Let $(\cdot,\cdot)_{\Omega(t)}$ be the inner product on $\Omega(t)$. Using the Galerkin orthogonality (Prop.~\ref{prop:orthogonality}), we have 
\[
	B^n_h(e_h,v_h) = -B^n_h(e_p,v_h), \textrm{ for all } v_h \in V^{n}_{h,ag}.
\]
Integration by parts (see Prop.~\ref{prop:time_dep_terms_equality}) yields
\begin{align*}
	B^n_h( e_p,v_h) &= \int_{J^{n}}^{} (( \partial_t e_p,v_h)_{\Omega(t)} + a(t,e_p,v_h)) \mathrm{~d}t + (e_{p}^{n-1,+} - e_{p}^{n-1,-},v_{h}^{n-1,+})_{\Omega^{n-1}} \\
	&= \int_{J^{n}}^{} (-( e_{p},\partial_t v_{h})_{\Omega(t)} + a(t,e_p,v_h)) \mathrm{~d}t + \int_{\partial Q^n_N \cup \partial Q^n_D}n_te_p v_h \mathrm{~d}S 
	\\ &\quad + (e_{p}^{n,-},v_{h}^{n,-})_{\Omega^n} - (e_{p}^{n-1,-},v_{h}^{n-1,+})_{\Omega^{n-1}}.
\end{align*}
As, $\partial_t v_{h} \in \bar{V}_{h,ag}^n \otimes \mathcal{P}_{q-1}(J^n)$, using \eqref{eq:local_projector_fem}, we have
\[
\int_{J^n}^{} ( e_{p},\partial_t v_{h})_{\Omega(t)}\mathrm{~d}t = 0,
\]
and, since $v_{h}^{n,-} \in \bar{V}_{h,ag}^n$, using the definition of the space-only projector, we get 
\[
(e_{p}^{n,-},v_{h}^{n,-})_{\Omega^{n}} = (u(t^{n})^{-}-\bar{\pi}^n_{h}(u(t^{n}))^{-},v_{h}^{n,-})_{\Omega^{n}} = 0. 
\] 
Therefore,
\begin{align*}
	B^n_h(e_p,v_h) = \int_{J^n}^{} a(t,e_p,v_h)) \mathrm{~d}t + \int_{\partial Q^n_n \cup \partial Q^n_D}n_te_p v_h \mathrm{~d}S - (e_{p}^{n-1,-},v_{h}^{n-1,+})_{\Omega^{n-1}}.
\end{align*}
Setting $v_h = e_h$, and using Prop~\ref{prop:stability_st}, we get
	\begin{align*}
		& \tnor{e_h}_n^2 + c_\mu \|e_h\|^2_{L^2(Q^n)} \lesssim  \| e^{n-1,-}_{h} \|_{L^2(\Omega^{n-1})}^2 + B_h^n(e_h,e_h) \\
		&\lesssim \| e^{n-1,-}_{h} \|_{L^2(\Omega^{n-1})}^2 -\int_{J^n}^{} a(t,e_p,e_h) \mathrm{~d}t - \int_{\partial Q^n_N \cup \partial Q^n_D}n_te_p e_h \mathrm{~d}S + (e_{p}^{n-1,-},e_{h}^{n-1,+})_{\Omega^{n-1}}.
	\end{align*}
	Choosing $\beta_{\bar{T}}>1$, the continuity in Prop.~\ref{prop:coercivity_space} and Young's inequality, we get 
	\begin{align*}
		\tnor{e_h}_n^2 + c_\mu \|e_h\|^2_{L^2(Q^n)} \lesssim & \| e^{n-1,-}_{h} \|_{L^2(\Omega^{n-1})}^2 + \int_{J^n}^{}\|e_p\|^2_{\bar{V}(h)} \mathrm{~d}t +\frac{c_\mu}{4}\int_{J^n}^{} \|e_h\|^2_{\bar{V}(h)} \mathrm{~d}t \\ & - \int_{\partial Q^n_N}n_te_p e_h \mathrm{~d}S+ (e_{p}^{n-1,-},e_{h}^{n-1,+})_{\Omega^{n-1}}.
	\end{align*}
	We can estimate the term on the Neumann boundary using Cauchy-Schwarz and Young's inequalities and \eqref{eq:trace_ineq_agg} as 
	\begin{align*}
		 \left|\int_{\partial Q^n_N}n_te_p e_h \mathrm{~d}S \right| &\lesssim \int_{J^n} \zeta^{-1}\|e_p\|_{L^2(\partial \Omega_N(t))}^2 + \zeta \|e_h\|_{L^2(\partial \Omega_N(t))}^2 \mathrm{~d}t\\
		&\lesssim \int_{J^n}  \zeta^{-1}\|e_p\|_{L^2(\Omega(t))} \|e_p\|_{H^1(\Omega(t))} +  \zeta \|e_h\|_{L^2(\Omega(t))}\|e_h\|_{H^1(\Omega(t))}\mathrm{~d}t\\
		&\lesssim \int_{J^n}\zeta^{-1}\|e_p\|_{H^1(\Omega(t))}^2 \mathrm{~d}t+ \int_{J^n} \zeta  \|e_h\|_{H^1(\Omega(t))}^2\mathrm{~d}t.
	\end{align*}
	Choosing $\zeta$ small enough, we get
	\begin{align*}
		\tnor{e_h}_n^2  &- \frac{c_\mu}{2}\int_{J^n} \|e_h\|^2_{\bar{V}(h)} \mathrm{~d}t \\&\lesssim \| e^{n-1,-}_{h} \|_{L^2(\Omega^{n-1})}^2 + \int_{J^n}\|e_p\|^2_{\bar{V}(h)} \mathrm{~d}t  + (e_{p}^{n-1,-},e_{h}^{n-1,+})_{\Omega^{n-1}}.
	\end{align*}
	The last term on the RHS can be approximated using the fact that $e_{h}^{n-1,-} \in \bar{V}_{h,ag}^{n-1} \text{ defined on } \Omega^{n-1}$:
	\begin{align*}
		(e_{p}^{n-1,-},e_{h}^{n-1,+})_{\Omega^{n-1}} &= (e_{p}^{n-1,-},e_{h}^{n-1,+}- e_{h}^{n-1,-})_{\Omega^{n-1}}\\
		&\le \frac{1}{2}\|e_{p}^{n-1,-}\|^2_{L^2(\Omega^{n-1})} + \frac{1}{2}\|e_{h}^{n-1,+}- e_{h}^{n-1,-}\|^2_{L^2(\Omega^{n-1})}.
	\end{align*}
\end{proof}

\begin{remark}
	We would like to point out that if we use an $H^{(1,0)}(Q^n)$ stable space-time projector $\pi_{h\tau}$ in Prop.~\ref{prop:proj_error_vh}, then we can avoid mixed terms and end up in an estimate of the form 
	\begin{equation*}
		\int_{J^n} \|u -\pi^n_{h\tau}(u) \|^2_{\bar{V}(h)} \mathrm{~d} t \lesssim  h^{2p}|u|^2_{H^{(p+1,0)}(Q^n_{act})} + \tau^{2(q+1)}|u|_{H^{(1,q+1)}(Q^n_{act})}^2.
	\end{equation*}
	However, the $H^{(1,0)}(Q^n)$ stable space-time projector would not satisfy \eqref{eq:local_projector_fem}. Therefore, we cannot eliminate the term involving the time derivate in the proof of Prop.~\ref{prop:approx_error_bounds}.
\end{remark}

\begin{theorem}[Estimates for total error in the \ac{dg} norm]\label{thm:total_error_estimate}
  Let $u_h$ be the solution of \eqref{eq:whole-problem} for $\mathcal{V}_{h,ag}$ of order $p$ in space and $q>0$ in time. Let us assume that the solution $u$ of \eqref{eq:heat_eqn} can be extended to $Q_{act}$ in such a way that $u \in H^{(p+1,q+1)}(Q_{act})$. The following bound for the total error holds:
	\begin{equation}\label{eq:total_err_estimate}
		\begin{aligned}
			\tnor{e}^2_{N,*} \lesssim h^{2p}|u|^2_{H^{(p+1,0)}(Q_{act})} + h^{-2}\tau^{2(q+1)}|u|_{H^{(0,q+1)}(Q_{act})}^2+ h^{2(p+1)} \sup \limits_{0 \le t \le T}|u|_{H^{p+1}(\Omega_{act}(t))}^2.
		\end{aligned}
	\end{equation}
\end{theorem}

\begin{proof}
	The total error $e$ can be expressed as sum of approximation and projection errors, that is, $e = e_p + e_h$. So,
	\begin{align*}
		\tnor{e}^2_{N,*} \le \tnor{e_p}^2_{N,*} + \tnor{e_h}^2_{N,*}.
	\end{align*}
	From Prop~\ref{prop:approx_error_bounds}, we have 
	\begin{align*}
		\tnor{e_h}_{N,*}^2 \lesssim \|e_{h}^{0,-}\|^2_{L^2(\Omega^{0})}&+\int_{0}^{T} \|e_p\|_{\bar{V}(h)}^2 \mathrm{d}t + \sum_{n=0}^{N-1}  \|e_{p}^{n,-}\|^2_{L^2(\Omega^{n})}.
		\end{align*}
		Setting $u_{h}^0 = \bar{\pi}^0_{h}(u^0)$, we get $e_{h}^{0,-} = u_{h}^{0,-} -  \bar{\pi}^0_{h}(u^0)^{-} = 0$. Using Prop.~\ref{prop:proj_error_time_points} and \ref{prop:proj_error_vh}, we get
		\begin{align*}
			\tnor{e_h}_{N,*}^2 \lesssim h^{2p}|u|^2_{H^{(p+1,0)}(Q_{act})} + h^{-2}\tau^{2(q+1)}|u|_{H^{(0,q+1)}(Q_{act})}^2+ h^{2(p+1)} \sup \limits_{0 \le t \le T}|u|_{H^{p+1}(\Omega_{act}(t))}^2.
		\end{align*}
		This result together with Propositions~\ref{prop:proj_error_time_points} and \ref{prop:proj_error_vh} proves the desired estimate.

 \end{proof}

\begin{remark}
      Let us compare the a priori error bounds above with similar analyses in the literature. The analysis in \cite{Lehrenfeld2013} for an \ac{xfem}-\ac{dg} unfitted method makes use of essentially the same norm as above for the discretisation error but a stronger norm is required for the approximation error. Besides, the a priori error bounds in \cite{preuss_ma,heimann_ma} for a face-based ghost penalty unfitted discretisation are obtained for a stronger upwind-like norm that also provides control on the time derivative of the solution. 
\end{remark}

\subsection{Solving the discrete system}\label{sec:cond_num_analysis}
\par 

In this section, we analyse the conditioning of the system matrix to be inverted at every time slab. The main motivation behind the proposed formulation is the ability to yield matrices that are well-posed independently of the cut location. Some previous analyses on the robustness of \ac{agfe} methods are obtained for self-adjoint operators \cite{Badia2018aggregated}. This is not the case for the time-dependent problems considered in this work. The resulting linear system matrix is nonsymmetric. However, one can consider a preconditioning technique to end up with a preconditioned symmetric linear system. We refer the reader to \cite{Smears2016} for a preconditioning technique for \ac{dg} time discretisations and fixed spatial domains. We are not aware of preconditioning techniques for time-varying domains. It is not the aim of this section to design effective preconditioner for this more general case. Instead, following \cite{Smears2016}, we prove that such preconditioner is effective in the time-constant domain for the unfitted formulation proposed in this work. 

The effectiveness of the preconditioner in \cite{Smears2016} relies on the condition number of the following mass and stiffness matrices: 
\begin{equation}\label{eq:system_matrix_st}
	\boldsymbol{M}_{ab} = \int_{Q^n} \mathcal{E}^n(\Phi^a) \mathcal{E}^n(\Phi^b) \mathrm{~d}\boldsymbol{x}\mathrm{~d}t, \qquad
  \boldsymbol{A}_{ab} = \int_{Q^n} a_h(t,\mathcal{E}^n(\Phi^a),\mathcal{E}^n(\Phi^b)) \mathrm{d}\boldsymbol{x}\mathrm{~d}t, 
\end{equation}
$\text{ for } a,b \in \mathcal{N}_{h,in}^n$, the set of internal nodes in $J^n$.
We prove below that the condition numbers of these two matrices do not depend on the cut location.

\begin{proposition}[Space-time mass and stiffness matrix condition number]\label{prop:system_cond_num}
	The condition number of the system matrices $\boldsymbol{M}$ and $\boldsymbol{A}$ in \eqref{eq:system_matrix_st} associated with the \ac{agfe} space $V^n_{h,ag}$ are bounded by $C$ and  $Ch^{-2}$, resp., for some positive constant $C$. 
\end{proposition}

\begin{proof}
Let $\boldsymbol{u}$  denote the nodal vectors of $u_{h}$, which is constant in time. 
Combining the tensor product expression in \eqref{eq:extension_tensor_product} for the space-time extension operator and the continuity of the discrete extension operator for all time values in $J^n$(see, \cite[Lemma~5.2.]{Badia2018aggregated}), we get
	\begin{align}\label{eq:bounds_extension_st}
		 { \tau^n h^d \|\boldsymbol{u}\|^2_2 \lesssim \|\mathcal{E}^n(u_h)\|^2_{L^2(Q^n)} \lesssim \tau^n h^d  \|\boldsymbol{u}\|^2_2}, \quad \text{ for } u_h \in V^n_{h,in}.
	\end{align}
Using the coercivity of the spatial bilinear form for all $t \in (0,T]$ in Prop.~\ref{prop:coercivity_space}, \eqref{eq:extension_tensor_product}, \cite[Lemma~5.8.]{Badia2018aggregated} and \eqref{eq:bounds_extension_st}, we have
	\begin{equation}\label{eq:stiff_lower_bound}
		\int_{J^n} a_h(t,\mathcal{E}^n(u_h),\mathcal{E}^n(u_h))\mathrm{~d}t \gtrsim \int_{J^n} \left(\|\mathcal{E}^n(u_h)\|^2_{\bar{V}_h} + \|\mathcal{E}^n(u_h)\|_{L^2(\Omega(t))}^2\right) \mathrm{~d}t \gtrsim {\tau^n h^d \|\boldsymbol{u}\|^2_2}.
	\end{equation}
Using \eqref{eq:extension_tensor_product} and \cite[Cor.~5.9.]{Badia2018aggregated}, we get
	\begin{equation}\label{eq:stiff_upper_bound}
		\int_{J^n}a_h(t,\mathcal{E}^n(u_h),\mathcal{E}^n(u_h))\mathrm{~d} t \lesssim {\tau^n h^{d-2} \|\boldsymbol{u}\|^2_2}.
	\end{equation}
It proves the proposition.
\end{proof}

\section{Numerical experiments}\label{sec:num_expts}

\subsection{Implementation remarks}\label{sec:implementation_remarks}
 
The numerical examples below have been computed using the Julia programming language \cite{Bezanson2017} (version 1.7.1) and several components of the \texttt{Gridap} project \cite{Badia2020gridap} (version 0.17.12), which is a free and open-source \ac{fem} framework written in Julia. \texttt{Gridap} combines a high-level user interface to define the weak form in a syntax close to the mathematical notation and a computational backend based on the Julia JIT-compiler, which generates high-performant code tailored for the user input \cite{Verdugo2022}. We have taken advantage of the extensible and modular nature of \texttt{Gridap} to implement the new methods in this paper. In particular, we have heavily used \texttt{GridapEmbedded}  \cite{GridapEmbedded-jl} (version 0.8.0), a plug-in that provides  functionality to implement different types of embedded \ac{fe} methods, including the generation of embedded integration meshes from level-set functions and different types of stabilisation schemes based on ghost-penalty and \ac{agfem}. 

The computational code developed to run the examples consists of a temporal loop over all time slabs. At each time slab, the discrete weak problem \eqref{eq:weak_form_st} is solved. The assembly of the linear system associated with equation \eqref{eq:weak_form_st} and its solution can be performed with the high-level tools provided by \texttt{Gridap} and \texttt{GridapEmbedded}  plus minor non-intrusive extensions. The generation of the background space-time mesh $\mathcal{T}^n_{h,art}$ for a given time slab $J^n$ and its intersection with the space-time level-set function describing $Q$ can be readily done using the available functionality in \texttt{Gridap} and \texttt{GridapEmbedded} , but applied to $d+1$ dimensions. This is also true for the space-time interpolation space $V^n_{h,ag}$ and the numerical integration of quantities in the space-time integration domains $Q^n\cap T$ for $T\in \mathcal{T}^n_{h,act}$. In fact, these operations can be done with any \ac{fem} code that supports \ac{agfem}-based embedded computations on $d+1$ spatial dimensions. With \texttt{Gridap} and \texttt{GridapEmbedded}, one could even consider  $d=3$ (thus involving  space-time domains in $d+1=4$ dimensions) since most of the code is implemented for an arbitrary value of $d$. In any case, we have restricted the following numerical experiments to $d=2$ due to computational power constraints. In the future, we plan to combine the implementation with \texttt{GridapDistributed} \cite{alberto_f_martin_2022_6076710} to exploit parallel environments.

The main extensions we need to solve the weak form \eqref{eq:weak_form_st} are the following. On one hand, the time derivative $\partial_t$ and the \emph{space-only} gradient operator $\nabla$ are not available in \texttt{Gridap} for functions in $d+1$ dimensions since they are specific of space-time methods. However, the implementation of these operators is just a post-process of the full gradient operator provided by \texttt{Gridap} (i.e., the standard gradient containing all partial derivatives w.r.t. the $d+1$ coordinates). The most intricate extension is related with the numerical integration of the space-time shape functions $v_h\in V^n_{h,ag}$ in the spatial domain $\Omega(t^{n-1})$. To this end, we have implemented $f(t)$,  the restriction of a given space-time function $f$
to a given time instant $t$, returning a space-only function. To implement this operation, we assume for simplicity and without any loss of generality that $f$ is defined on the reference cell of the space-time mesh $\mathcal{T}^n_{h,art}$, since this is usually the case for the shape functions in $V^n_{h,ag}$. Function $f(t)$ can be conveniently introduced in the code as the function composition $f(t)\doteq f\circ \phi^n_t$, where $\phi^n_t$ is a map that goes from the reference cell of the space only mesh $\bar{\mathcal{T}}^n_{h,art}$ to the reference cell of the space-time mesh $\mathcal{T}^n_{h,art}$. For $d=2$ and $d=3$, the map $\phi^n_t$ is defined for a linear approximation of the geometry in time as
$$\phi^n_t(x,y) \doteq (x,y,\frac{t-t^{n-1}}{t^n-t^{n-1}})^\mathrm{T} \text{ and } \phi^n_t(x,y,z) \doteq (x,y,z,\frac{t-t^{n-1}}{t^n-t^{n-1}})^\mathrm{T},$$
respectively. In previous formulas, we have transformed the time $t$ in the ``physical'' domain $J^n=(t^{n-1},t^n)$ to a time value in the reference domain $(0,1)$. The definition of $\phi^n_t$ is analogous for other values of $d$. The implementation of $\phi^n_t$ is straightforward and the function composition $f\circ \phi^n_t$ can be readily computed with high-level tools available in \texttt{Gridap}. Using this functionality, one can easily compute  $v_h(t^{n-1})$ for the shape functions in $V^n_{h,ag}$ and then use the resulting space-only functions to compute the integrals on $\Omega(t^{n-1})$.

\subsection{Methods and parameter space}\label{sec:methods}

\par The numerical experiments in this section are designed to solve a heat equation with non-homogeneous boundary conditions that are weakly enforced using Nitsche's method. The Nitsche's coefficient is $\gamma = 10p(p+1)$, where $p$ is the order of spatial discretisation. The initial condition and source term are calculated such that the  manufactured solution \cite{Sudirham2006}
\begin{equation}\label{eq:man_sol}
	u(\boldsymbol{x},t) = \sin\left(\frac{\pi x}{L_{x}}\right)\sin\left(\frac{\pi y}{L_{y}}
	\right)\exp\left(\frac{-2\mu\pi^2 t}{T}\right)
\end{equation}
is an exact solution of \eqref{eq:heat_eqn}. Here, $\boldsymbol{x} =(x,y)$ is the 2-dimensional spatial variable. The lengths of the bounding box in the spatial dimensions are $L_{x}$ and $L_{y}$ respectively. The spatial background Cartesian mesh, $\Omega_{art} = [0,L_{x}] \times [0,L_{y}] = [0,2]\times[0,1]$. The final time is $T = 1$ and the diffusion coefficient is $\mu = 1$. We consider a moving geometry with a circular hole for the condition number tests and a moving geometry with a square hole for the convergence tests. The moving geometry with a circular hole is defined as
\begin{equation}\label{eq:moving_disk}
	Q_c =  \Omega_{art}\times[0,T] \setminus \{(\boldsymbol{x},t) : (x-1.5L_{x}+0.5L_{x}t)^2 
	+ (y-0.5L_{y})^2 \le 0.2^2\},
\end{equation}
and the moving geometry with a square hole is defined as
\begin{align}\label{eq:moving_square}
	Q_s =  \Omega_{art}\times[0,T] \setminus \{(\boldsymbol{x},t) : 1.5L_{x}+0.5L_{x}t-0.2
	&\le x \le 1.5L_{x}+0.5L_{x}t+0.2,\nonumber \\
	0.5L_{y}-0.2 &\le y \le 0.5L_{y}+0.2 \}.
\end{align}

\par We use Lagrangian reference \acp{fe} with bi-linear and bi-quadratic continuous polynomials in space and linear and quadratic discontinuous polynomials in time. We have not explored higher-order functional and geometrical discretisations in this work. We refer to \cite{Badia2022robust} for a recent extension of \ac{agfem} methods to higher-order discretisations.

\par The \texttt{cond()} function in Julia is used to evaluate the condition numbers. The condition numbers are computed in the 1-norm for efficiency reasons. 
The numerical experiments have been carried out at the MonARCH cluster at Monash University. 

\subsection{Condition number tests}\label{sec:condition_num_test}

\par In the first experiment, we move the centre of the disk along the $x$-axis, and calculate the condition numbers of the mass and stiffness matrices. We consider a spatial background mesh of size $h = 2^{-5}$ and a single time slab of size $\tau = 10^{-3}$. The position of the centre is perturbed as $(1.5L_{x}-0.5L_{x}t - \ell,0.5L_{y})$, where $0 \le \ell \le 1$.  We use different values for $\ell$ and calculate the condition numbers using \ac{sfem} and \ac{agfem} for linear and quadratic polynomials in space and time. We consider a polyhedral approximation of the disk and impose the exact boundary conditions on the approximated domain. As a result, we are not incurring in integration error.

\begin{figure}
	\centering
	\begin{subfigure}[b]{0.45\textwidth}
		\centering
		\includegraphics[width=\textwidth]{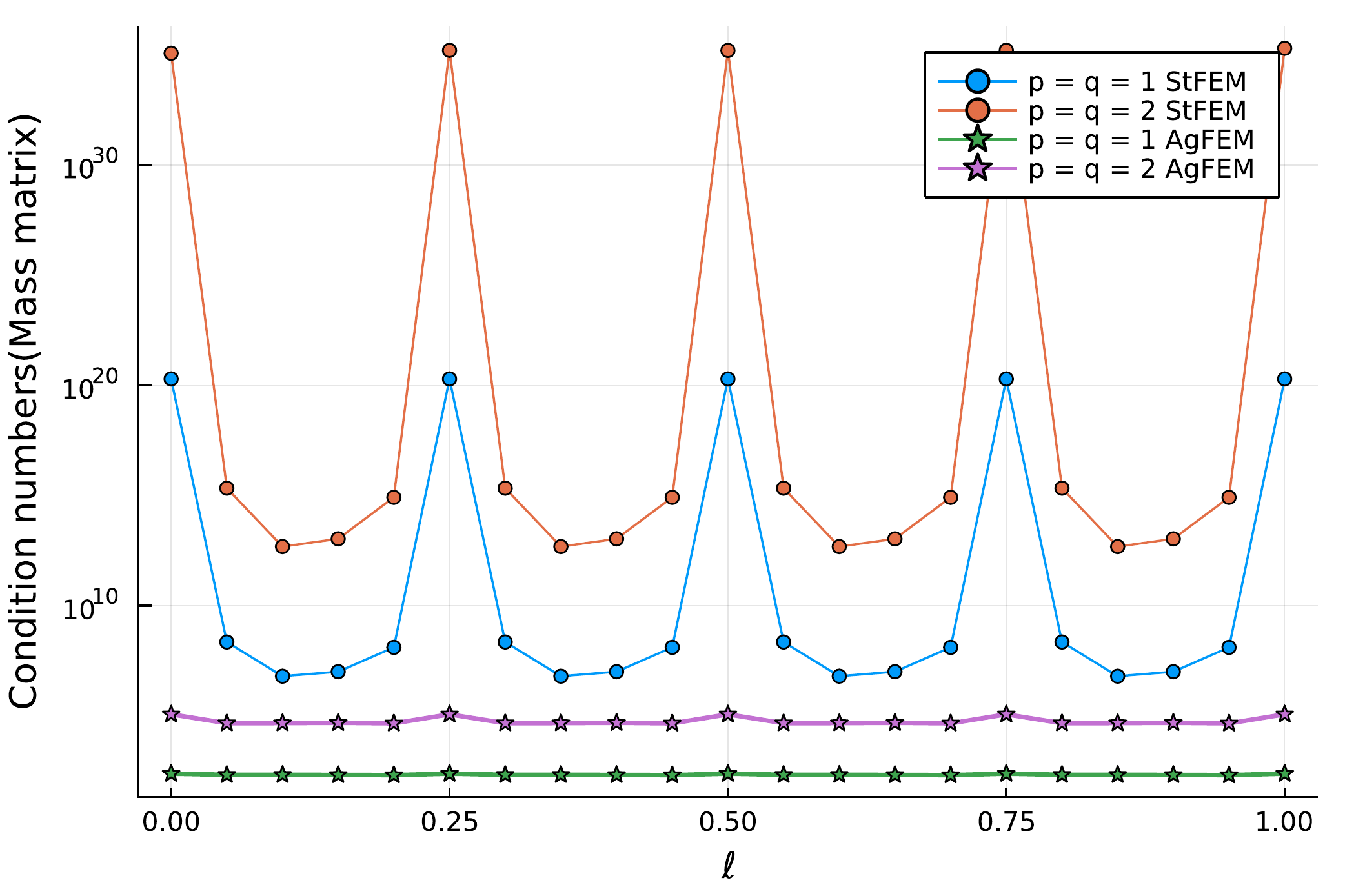}
		\caption{Mass matrix}
		\label{fig:robust_mass}
	\end{subfigure}
	\vspace{0.2cm}
	\begin{subfigure}[b]{0.45\textwidth}
		\centering
		\includegraphics[width=\textwidth]{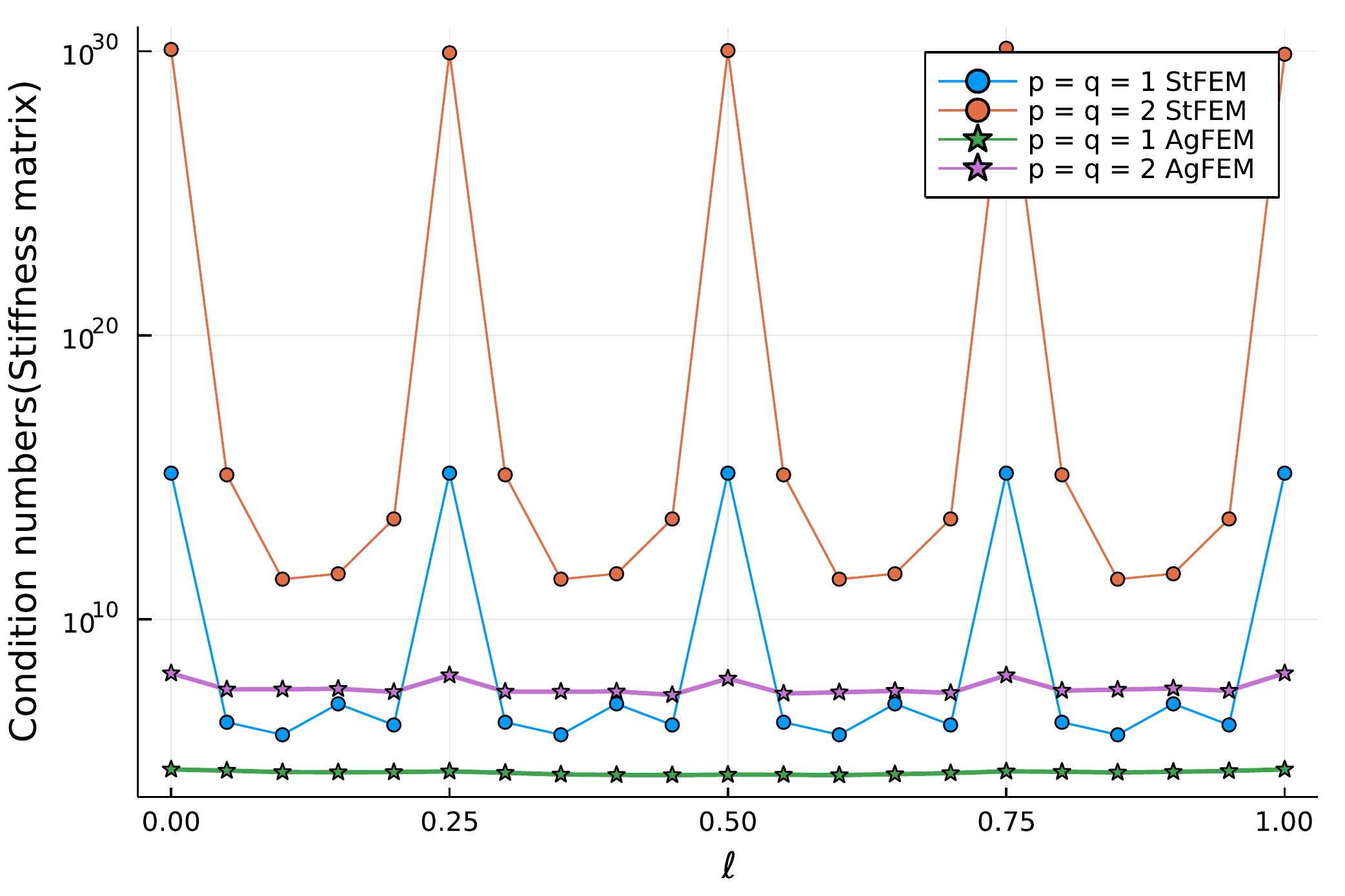}
		\caption{Stiffness matrix}
		\label{fig:robust_stiff}
	\end{subfigure}
	   \caption{The figures depict the plots of condition numbers of the mass and stiffness matrices respectively against perturbation of the centre of the disk 
	   	using \ac{agfem} and \ac{sfem} for linear and quadratic polynomials in space and time.}
	   \label{fig:robust_plots}
\end{figure}

The plot of condition numbers of the mass and stiffness matrices against the perturbation of the centre of the disk ($\ell$) is illustrated in Figure ~\ref{fig:robust_plots}. We observe that the condition numbers using \ac{agfem} are not affected by moving the position of the disk, i.e., it is robust with respect to the cut location, whereas there are huge fluctuations using \ac{sfem}. As the position of the disk changes, the cut locations change and some configurations of the geometry result in higher condition numbers using \ac{sfem} due to the small cut cell problem. The problem is more severe for quadratic \ac{sfem} in space-time, leading to almost singular matrices in some cases.  On the other hand, the position of the geometry plays a negligible role in determining the condition numbers of the mass and stiffness matrices in the proposed space-time \ac{agfem}.

\begin{figure}
	\centering
	\begin{subfigure}[b]{0.45\textwidth}
		\centering
		\includegraphics[width=\textwidth]{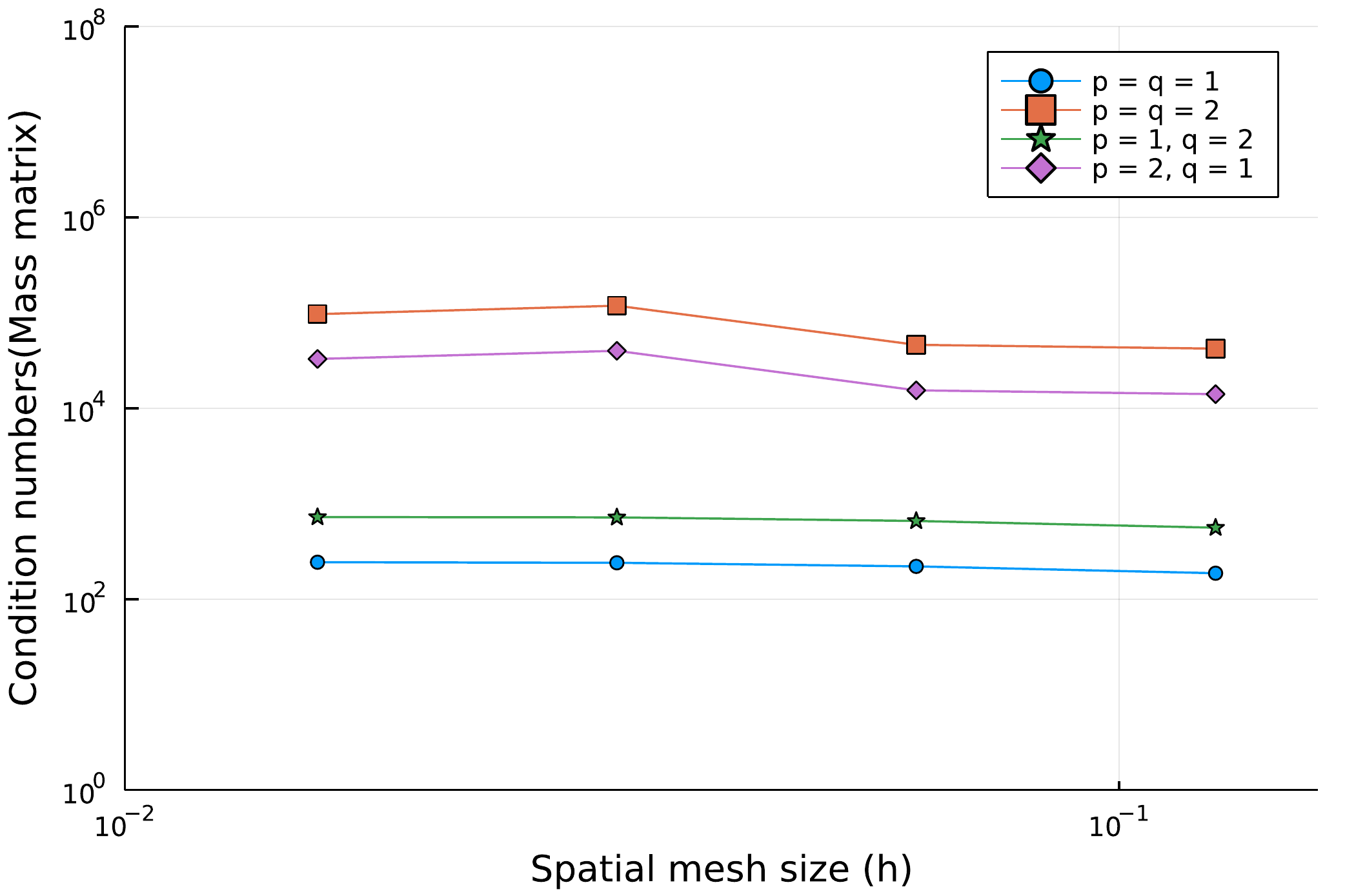}
		\caption{Mass matrix}
		\label{fig:cond_cgce_mass}
	\end{subfigure}
	\vspace{0.2cm}
	\begin{subfigure}[b]{0.45\textwidth}
		\centering
		\includegraphics[width=\textwidth]{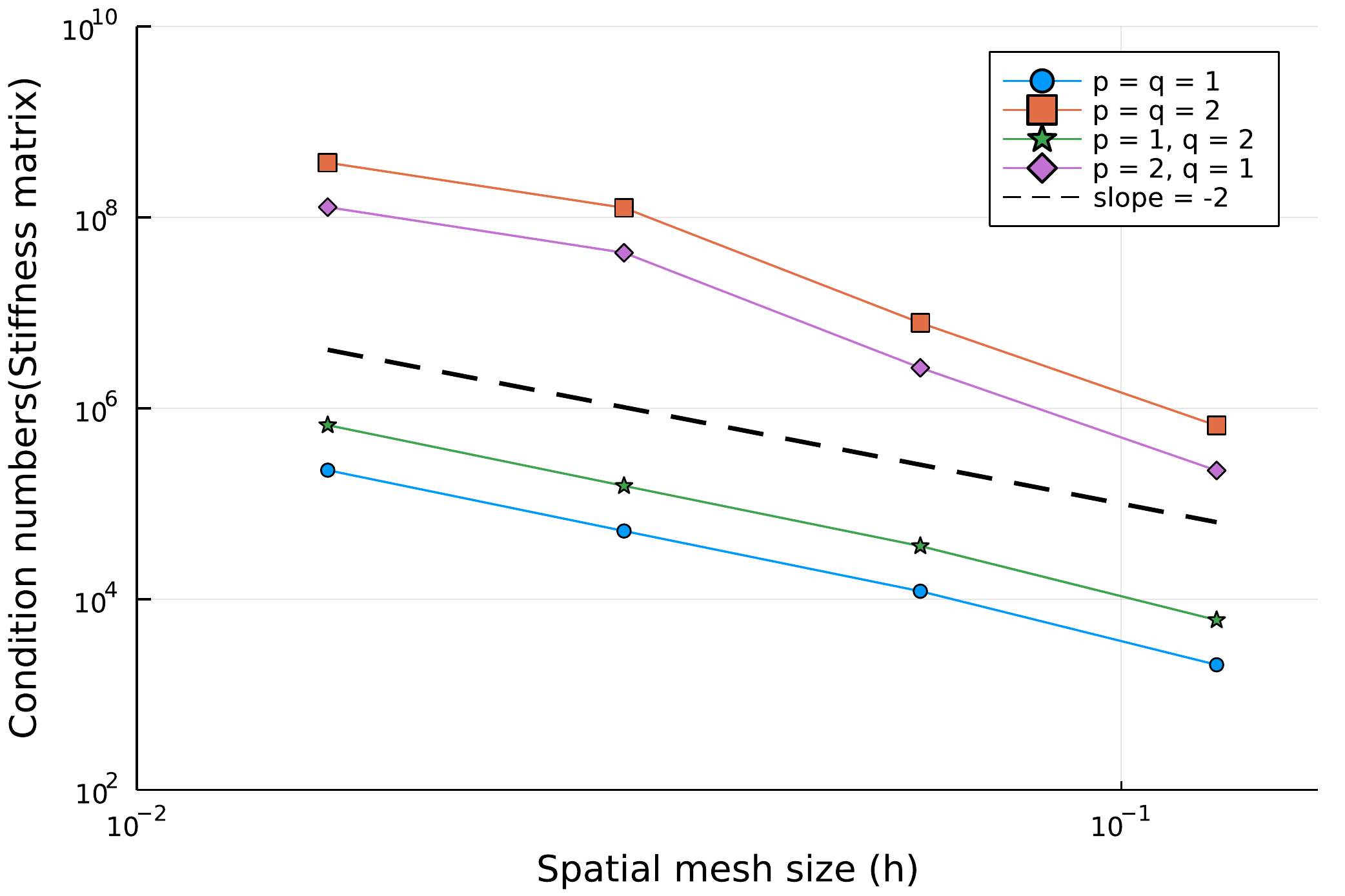}
		\caption{Stiffness matrix}
		\label{fig:cond_cgce_stiff}
	\end{subfigure}
	\caption{Plots of condition numbers of the mass and stiffness matrices against the spatial mesh size $h$
	using \ac{agfem} for a single time slab of size $\tau = h$.}
	\label{fig:cond_cgce_agfem}
\end{figure}

\par In the next experiment we study the behaviour of condition numbers of the mass and stiffness matrices with respect to mesh refinement. We consider the moving geometry with the circular hole and spatial background meshes of sizes $h = 2^{-m}, m = 3,4,5,6$ and a single time slab of size $\tau = h$. The plot of condition numbers against the spatial mesh size $h$ using \ac{agfem} is depicted in Figure~\ref{fig:cond_cgce_agfem}. We observe that the condition numbers of the mass matrix are almost constant whereas the condition numbers of the stiffness matrix scale with $\mathcal{O}(h^{-2})$ using \ac{agfem}, which is the expected ratio. 

\subsection{Convergence tests}\label{sec:convergence_test}

\par This experiment shows the behaviour of the error with respect to mesh refinement. We consider a geometry with a moving square hole, spatial background meshes of sizes $h = 2^{-m}, m = 3,4,5,6$ and a constant time step size $\tau = h$. Since the domain is polyhedral, the geometry is exactly represented. We plot the error in the accumulated \ac{dg} norm against the mesh size choosing the value of the coercivity constant $c_\mu = 1 $ in Figure~\ref{fig:err_dg_c=1}.
\begin{figure}
	\centering
	\begin{subfigure}[b]{0.45\textwidth}
		\centering
		\includegraphics[width=\textwidth]{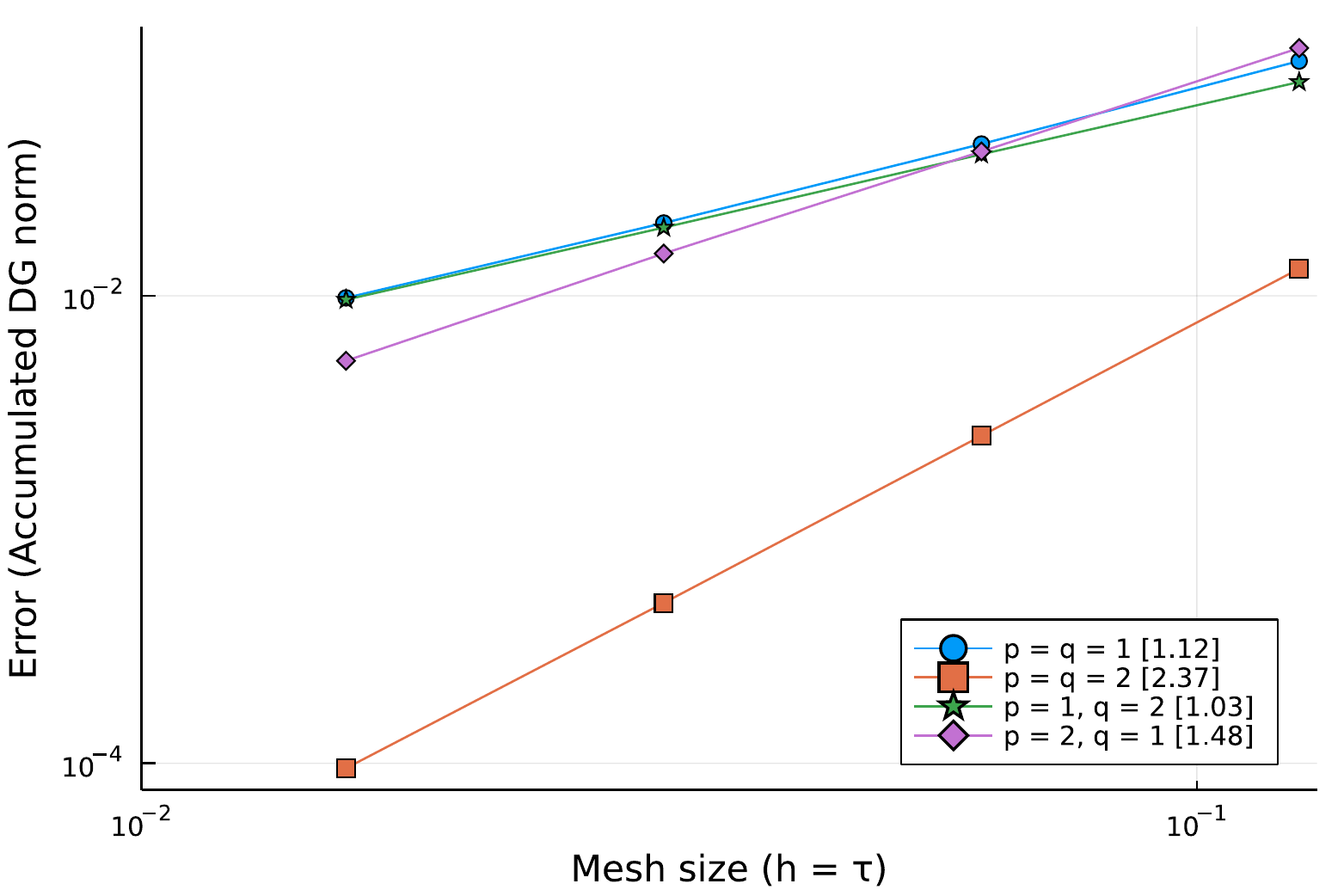}
		\caption{Accumulated \ac{dg} norm}
		\label{fig:err_dg_c=1}
	\end{subfigure}
	\vspace{0.2cm}
	\begin{subfigure}[b]{0.45\textwidth}
		\centering
		\includegraphics[width=\textwidth]{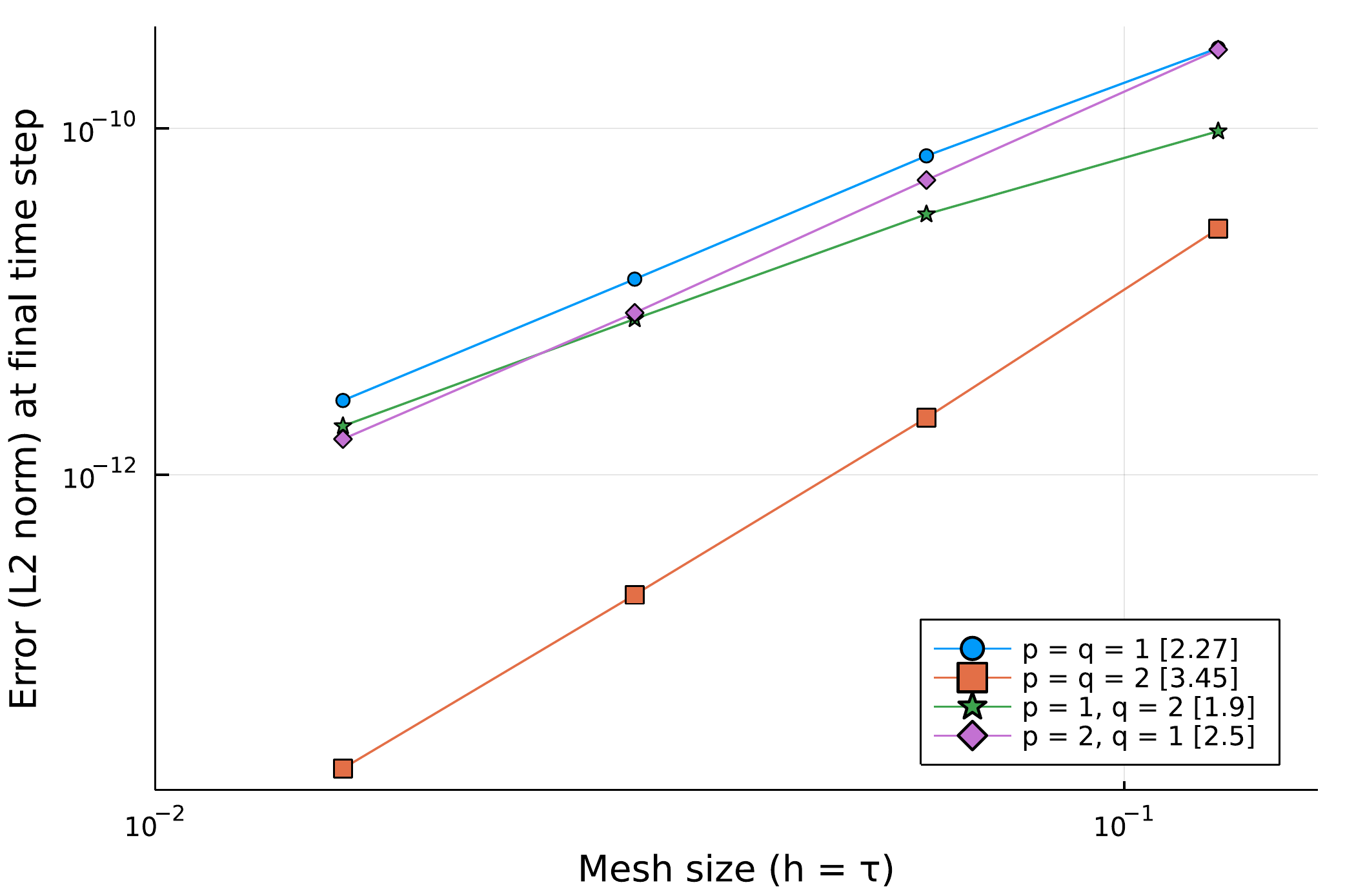}
		\caption{$L^2(\Omega^N)$ norm}
		\label{fig:err_l2_final}
	\end{subfigure}
	\caption{The first figure illustrates the plot of total error in the accumulated \ac{dg} norm (with constant $c_\mu = 1$) against the mesh size ($h=\tau$) using \ac{agfem}. The second figure depicts the plot of  error in the $L^2$ norm at the final time step against the mesh size ($h=\tau$) using \ac{agfem}}
	\label{fig:err_cgce_plots}
\end{figure}
We observe that using \ac{agfem}, when the ratio $h/\tau$ remains constant during refinement, the error converges with $\mathcal{O}(h^{s})$, where $s \doteq \min(p,q)$. This result is in agreement with \eqref{eq:total_err_estimate}.

\par In addition, with the same experimental setup the error in the $L^2(\Omega^N)$ norm is computed and plotted against the mesh size in Figure~\ref{fig:err_l2_final}. We observe higher convergence compared to the results using the accumulated \ac{dg} norm. The error scales with $\mathcal{O}(h^r)$, where $r\doteq\min(p+1,q+1)$.

\subsection{An example with topology change}
\label{sec:cdeq}

This last example studies the embedded space-time method in a more challenging geometrical configuration consisting of a time-dependent domain that undergoes topological changes. The example is taken from \cite{Lehrenfeld2019,preuss} where it is also considered to characterise the performance of other embedded \ac{fe} methods for time-evolving domains. The problem geometry is the union of two disks that travel with opposite velocities and eventually intersect (see Figure \ref{fig:cded1}). We describe the disks with the level-set functions,
\begin{equation*}
\begin{aligned}
\phi_1(\boldsymbol{x},t) = |\boldsymbol{x} - (0,t-3/4)^\mathrm{T}  | - 0.5,\\
\phi_2(\boldsymbol{x},t) = |\boldsymbol{x} - (0,3/4-t)^\mathrm{T}  | - 0.5,
\end{aligned}
\end{equation*}
$|\boldsymbol{v}|$ being the algebraic 2-norm of a vector $\boldsymbol{v}$. From these level-set functions, the time-dependent problem geometry is defined as
$$\Omega(t) = \{\boldsymbol{x}\in \mathbb{R}^2:\ \min(\phi_1(\boldsymbol{x},t),\phi_2(\boldsymbol{x},t))<0  \} \text{ for } t\in[0,T], $$
where the minimum operator is used to define the level-set function that describes the union of the two disks. The final time is selected as $T=3/2$ so that the initial and final geometry coincide, namely $\Omega(0) = \Omega(T)$. The geometry is implicitly defined via a level-set function that is linearly approximated.
\begin{figure}[ht!]

\begin{tiny}
$-1$\hspace*{.5em} \includegraphics[width=0.3\textwidth]{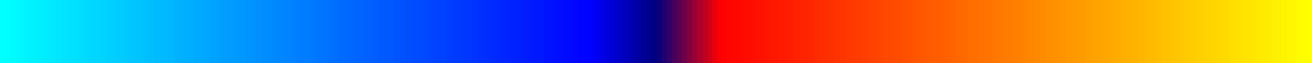} \hspace*{.5em} $1$
\end{tiny}

\begin{subfigure}{0.1\textwidth}
\adjincludegraphics[width=1\textwidth,trim={{0.18\width} 0 {0.18\width} 0},clip]{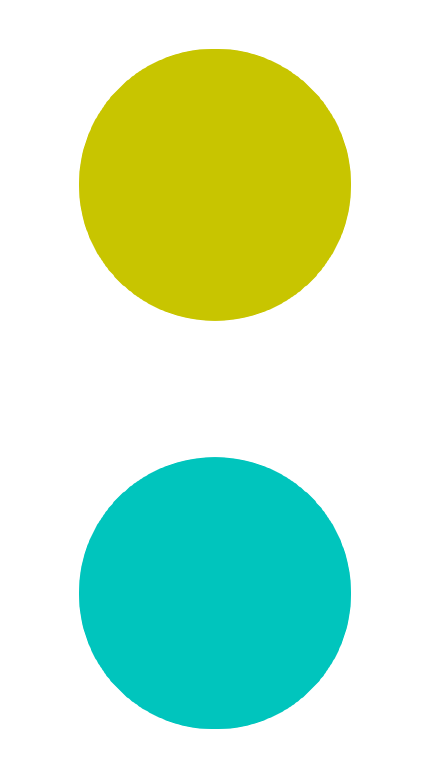}
\caption*{$t=0$}
\end{subfigure}
\begin{subfigure}{0.1\textwidth}
\adjincludegraphics[width=1\textwidth,trim={{0.18\width} 0 {0.18\width} 0},clip]{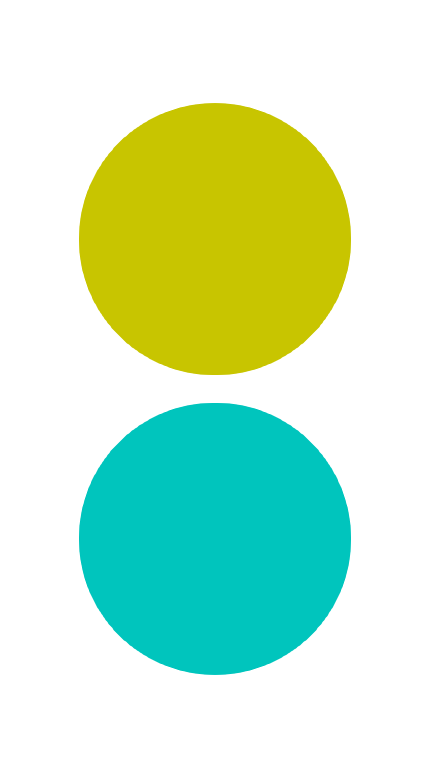}
\caption*{$t=.2$}
\end{subfigure}
\begin{subfigure}{0.1\textwidth}
\adjincludegraphics[width=1\textwidth,trim={{0.18\width} 0 {0.18\width} 0},clip]{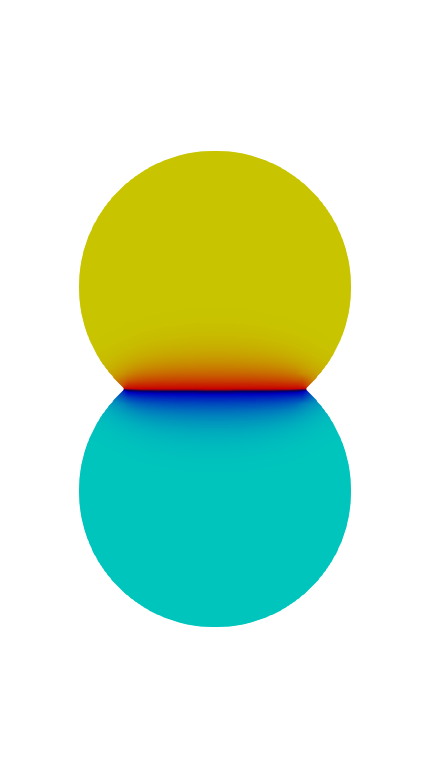}
\caption*{$t=.375$}
\end{subfigure}
\begin{subfigure}{0.1\textwidth}
\adjincludegraphics[width=1\textwidth,trim={{0.18\width} 0 {0.18\width} 0},clip]{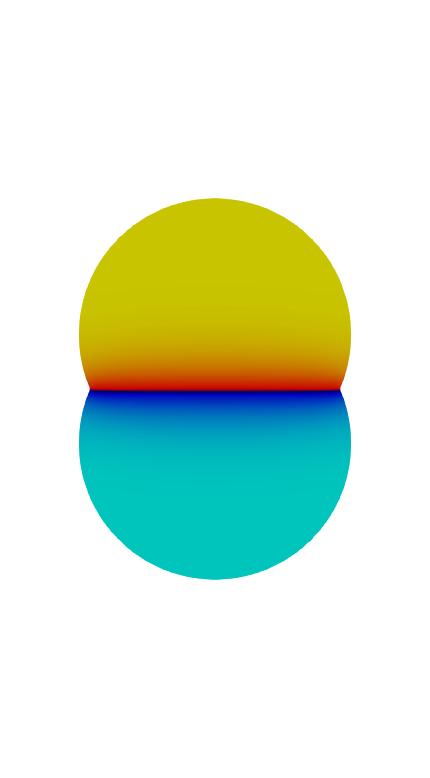}
\caption*{$t=.55$}
\end{subfigure}
\begin{subfigure}{0.1\textwidth}
\adjincludegraphics[width=1\textwidth,trim={{0.18\width} 0 {0.18\width} 0},clip]{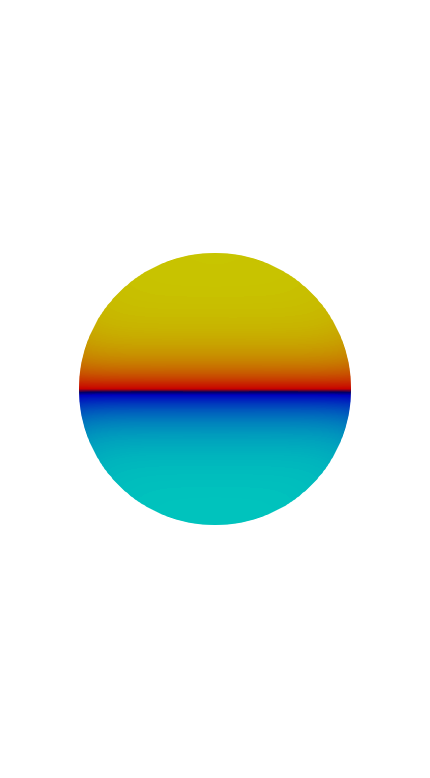}
\caption*{$t=.75$}
\end{subfigure}
\begin{subfigure}{0.1\textwidth}
\adjincludegraphics[width=1\textwidth,trim={{0.18\width} 0 {0.18\width} 0},clip]{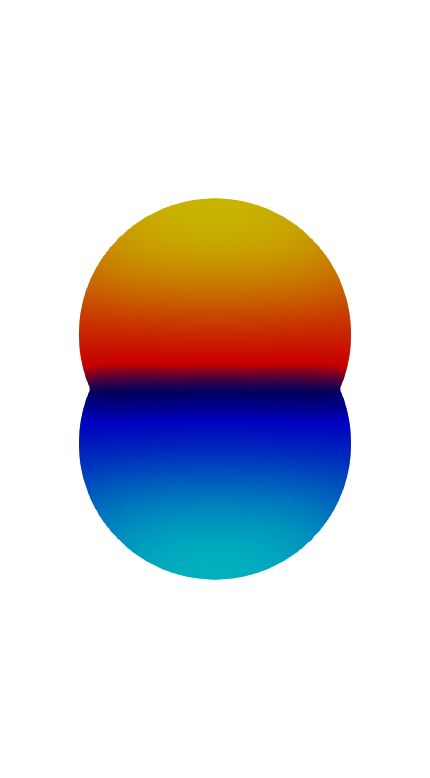}
\caption*{$t=.95$}
\end{subfigure}
\begin{subfigure}{0.1\textwidth}
\adjincludegraphics[width=1\textwidth,trim={{0.18\width} 0 {0.18\width} 0},clip]{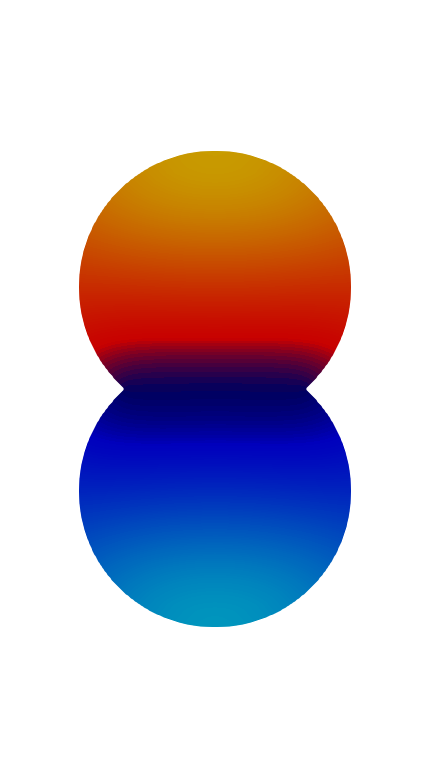}
\caption*{$t=1.125$}
\end{subfigure}
\begin{subfigure}{0.1\textwidth}
\adjincludegraphics[width=1\textwidth,trim={{0.18\width} 0 {0.18\width} 0},clip]{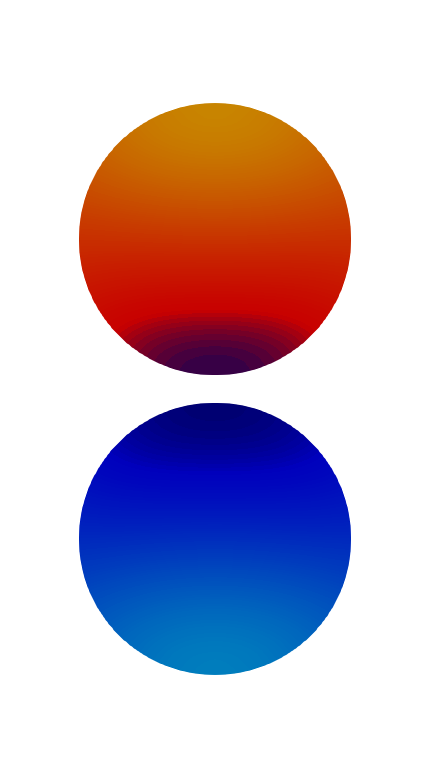}
\caption*{$t=1.3$}
\end{subfigure}
\begin{subfigure}{0.1\textwidth}
\adjincludegraphics[width=1\textwidth,trim={{0.18\width} 0 {0.18\width} 0},clip]{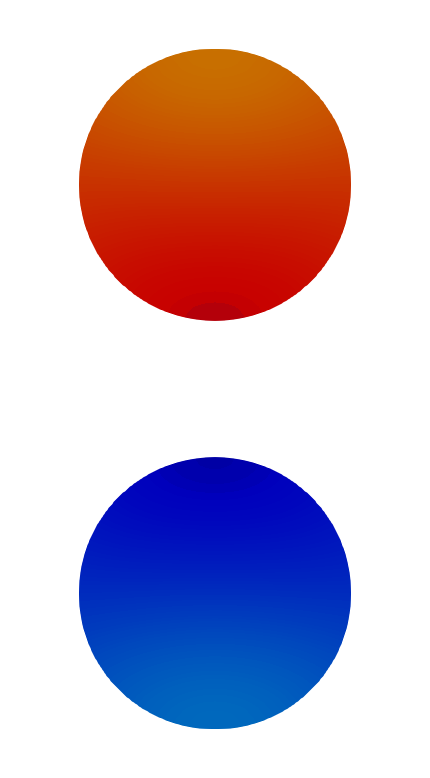}
\caption*{$T=1.5$}
\end{subfigure}

\caption{Numerical solution of the example in Section \ref{sec:cdeq} at nine selected time instants for a discretisation of $60\times121$ spatial cells and $60$ time slabs.}
\label{fig:cded1}
\end{figure}

On $\Omega(t)$, we solve the advection-diffusion equation $\partial_t u + \boldsymbol{w}\cdot\nabla u - \mu \Delta u = 0$ with homogeneous Neumann boundary conditions, $\boldsymbol{n_x}\cdot\nabla u = 0$, and the initial condition $u(x,y,0) = \mathrm{sign}(y)$. The advection velocity field is given by
\begin{equation*}
\boldsymbol{w}(x,y,t) = \left\lbrace
\begin{aligned}
(0,1)^\mathrm{T}\quad \text{if } y>0 \text{ and } t\leq T/2 \text{ or } y<0 \text{ and } t> T/2\\
(0,-1)^\mathrm{T}\quad \text{if } y\leq 0 \text{ and } t\leq T/2 \text{ or } y\leq 0 \text{ and } t> T/2\
\end{aligned}
\right..
\end{equation*}

 As in \cite{Lehrenfeld2019,preuss}, we take $\mu = 0.1$. For this value of the diffusion coefficient, the problem is diffusion-dominated and it can be solved with the numerical scheme presented in previous sections without any further stabilisation technique. We only need to introduce the advection term in the weak form in the obvious way. Adding numerical stabilisation for the advection term (e.g., SUPG) would be also possible, but we want to use a numerical scheme as close as possible as the one analysed in previous sections, as permitted by the diffusion-dominated nature of this example.

 As the advection velocity coincides with the motion of the disks $\boldsymbol{w}\cdot\boldsymbol{n_x} + n_t =0$ holds on the boundary of the domain and the problem is well-posed, as discussed in Sec.~\ref{sec:model_problem}.

For the numerical discretisation, we consider a Cartesian mesh of the artificial domain $\Omega_{art}\doteq(-0.6,0.6)\times(-1.35,1.35)$ with two different resolutions consisting of $60\times 121$ cells. We deliberately use an odd number of cells in the $y$-direction so that the first contact of the two disks happens within a single cut cell. Otherwise, the first contact would take place at a cell boundary, which is an unrealistically simple particular case. The temporal discretisation is fixed to $60$ time slabs.

Figure \ref{fig:cded1} shows the obtained solution. The constant initial condition at the two disks is transported by the advection field with the same velocity as the motion of the disks themselves. At the contact event, a topologically new domain is created and diffusion starts to take place due to a sudden formation of a concentration gradient. A detailed view of the contact zone is given in Figure \ref{fig:cded2}. Note that the numerical scheme is able to capture the sharp concentration gradient that takes place at the contact point without introducing numerical artefacts. In contrast to the results reported in \cite{Lehrenfeld2019,preuss}, we do not see any spurious diffusion starting before the time of first contact even though there is only a single full cell between the two disks at the time slab right before contact (see Figure \ref{fig:cded2-1}). For our formulation, diffusion would appear only when the disks get in touch (if aggregated cut cells are duplicated when they have disconnected regions). This is in contrast to the results in \cite{Lehrenfeld2019} for which the diffusion might start before, even with several layers of full cells between the disks, depending on the time step size. We can conclude that our numerical scheme is able to properly handle the topological change in this example. Similar results can be achieved for the weak \ac{agfem} method proposed in Sec.~\ref{sec:ghost_penalty}.

\begin{figure}[ht!]
\begin{subfigure}{0.32\textwidth}
\includegraphics[width=0.49\textwidth]{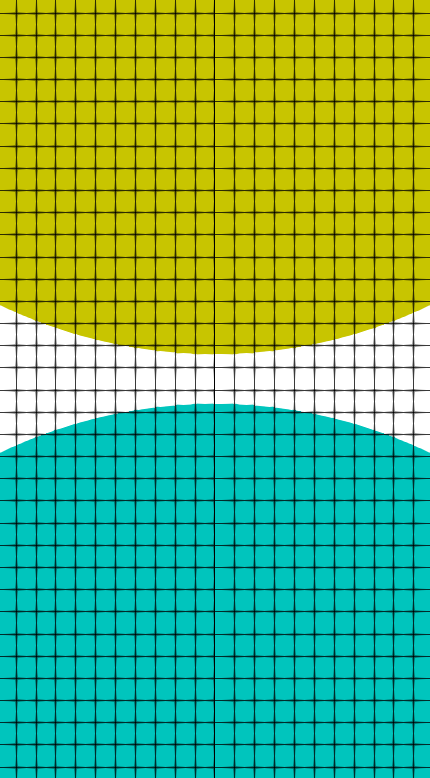}
\includegraphics[width=0.49\textwidth]{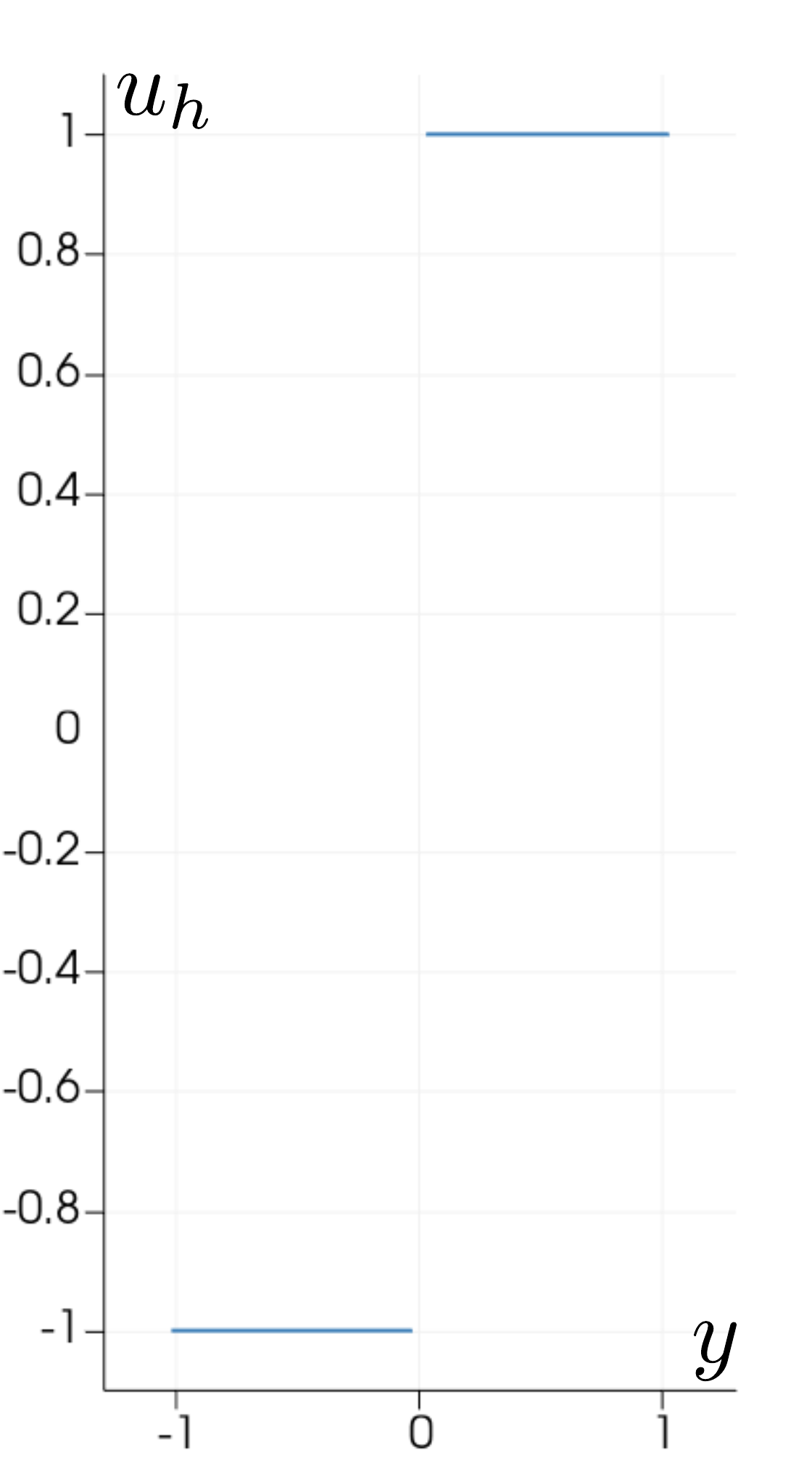}
\caption{$t=.225$}
\label{fig:cded2-1}
\end{subfigure}
\begin{subfigure}{0.32\textwidth}
\includegraphics[width=0.49\textwidth]{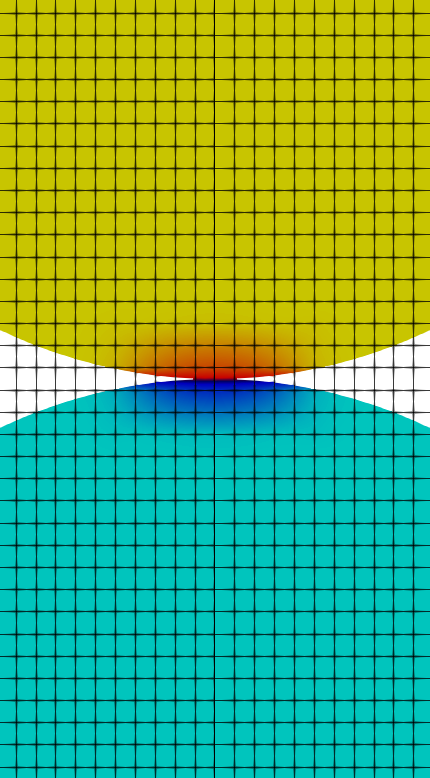}
\includegraphics[width=0.49\textwidth]{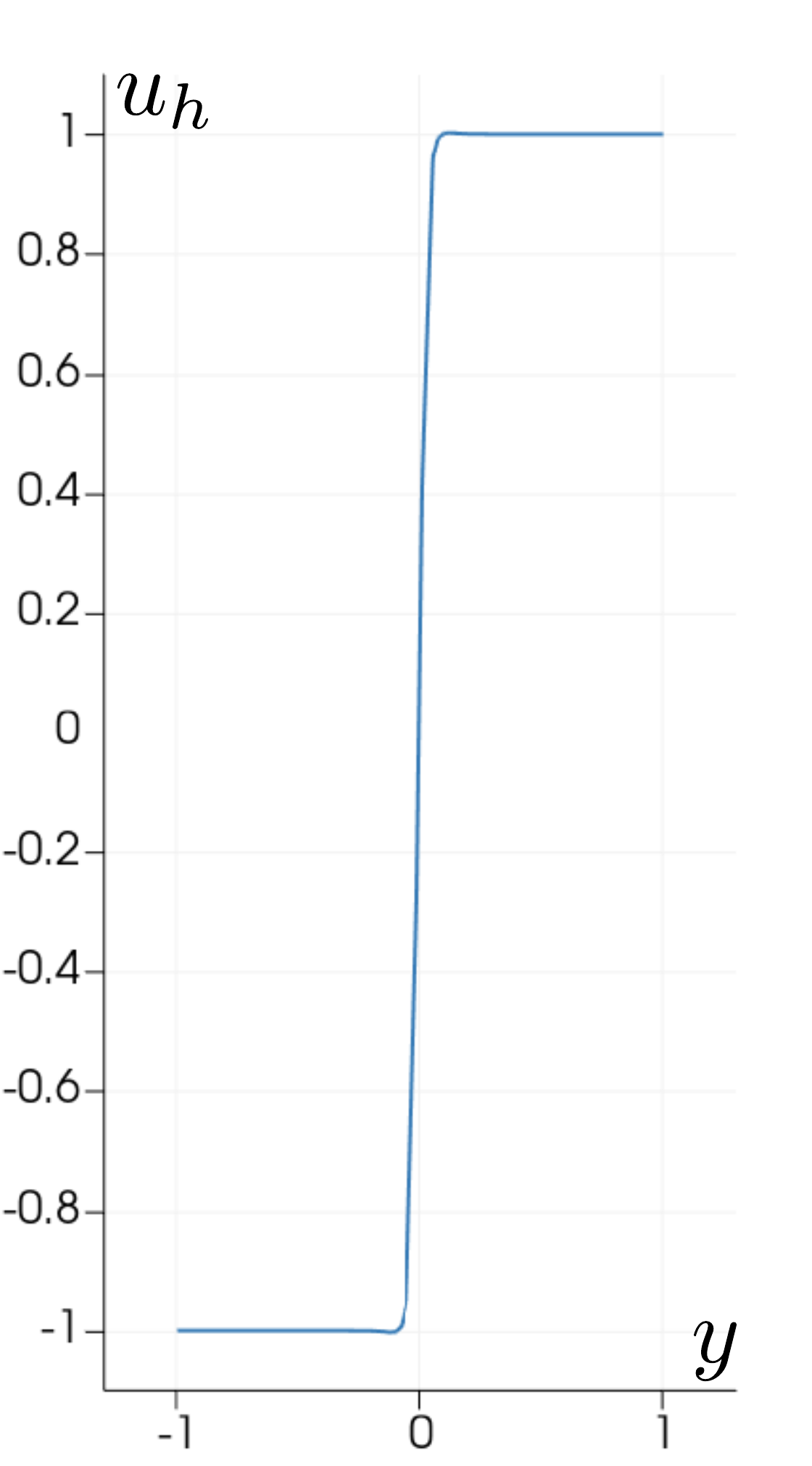}
\caption{$t=.25$}
\end{subfigure}
\begin{subfigure}{0.32\textwidth}
\includegraphics[width=0.49\textwidth]{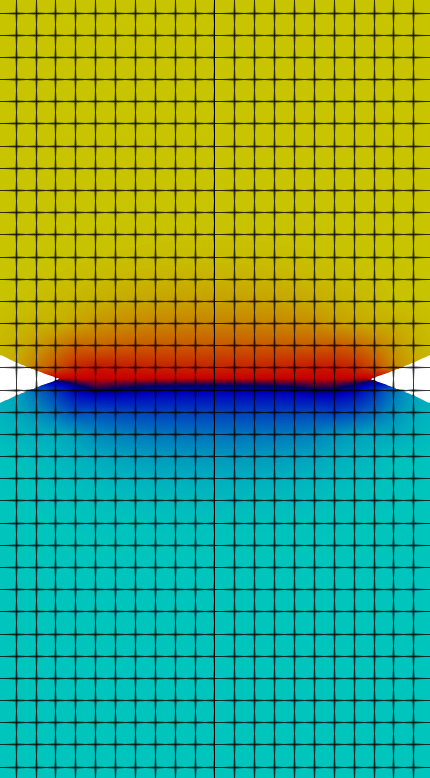}
\includegraphics[width=0.49\textwidth]{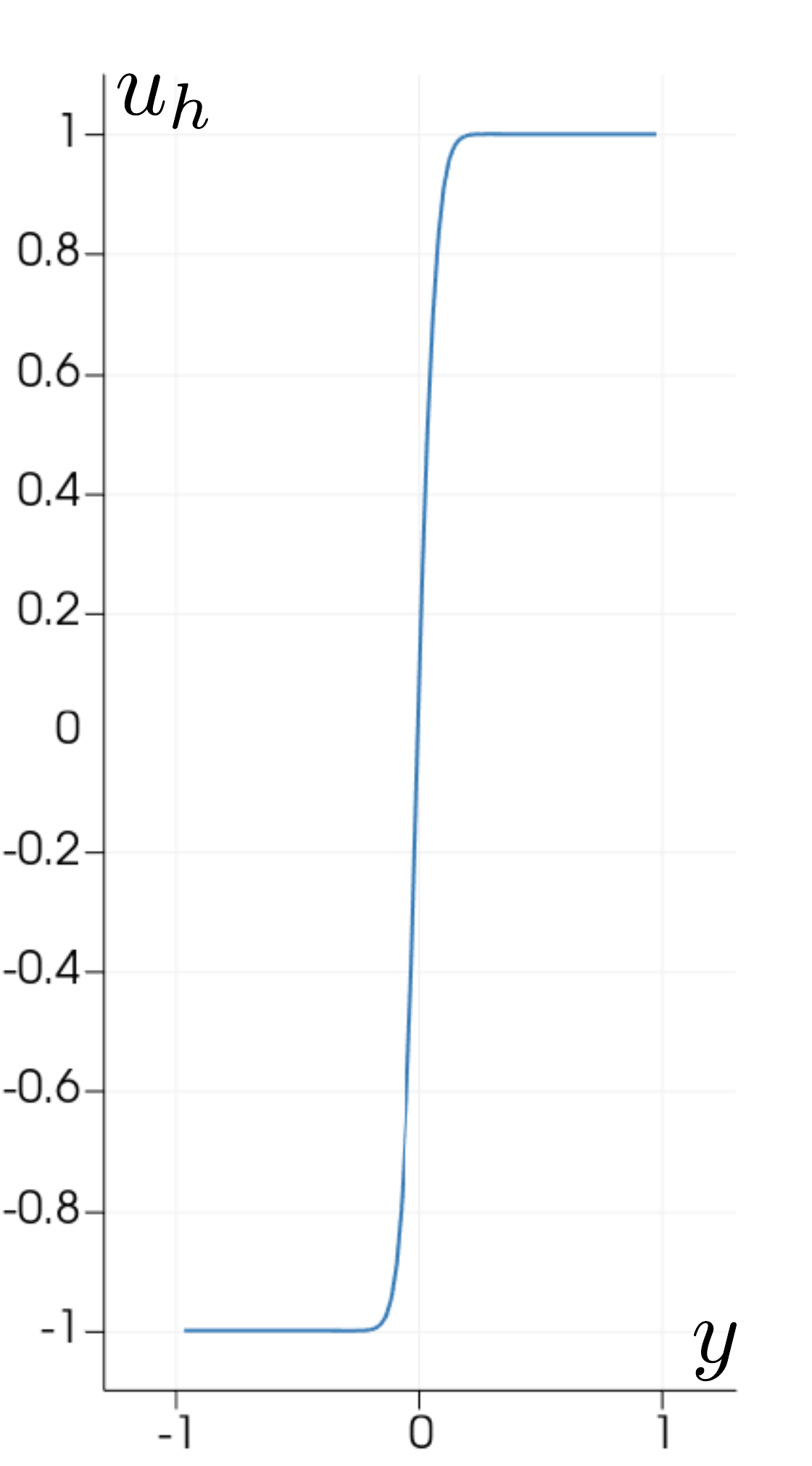}
\caption{$t=.275$}
\end{subfigure}
\caption{Detailed view of the contact zone for the last time step before contact, the time step when the contact takes place, and the next step after contact. The right hand side of each sub-figure shows the numerical solution $u_h$ restricted to the vertical line $x=0$ (the symmetry axis of the problem). The first contact point coincides with $(0,0)^\mathrm{T}$. The colour bar of this figure is the same as in Figure \ref{fig:cded1}.}
\label{fig:cded2}
\end{figure}

\section{Conclusions}\label{sec:conclusion}

\par In this work, we have proposed a novel space-time unfitted \ac{fe} technique to solve time-dependent \acp{pde}. The use of a variational space-time formulation is proposed to approximate problems with moving domains or interfaces. In order to circumvent the lack of robustness of these methods to cut locations, we have extended \ac{agfem} to space-time.

\par \ac{agfe} spaces are defined as the image of a discrete extension operator that constrains ill-posed \acp{dof} with well-posed \acp{dof}. Using a slab-wise time-constant cell aggregation algorithm, we have defined a discrete extension operator only in space at any time value. The image of this operator is a slab-wise \ac{agfe} space that can be expressed as a tensor-product of spatial and temporal spaces. Due to the definition of well-posedness of space-time cells, this discrete extension operator provides the required robustness with respect to the small cut cell problem.

\par We have carried out the numerical analysis (stability and convergence) of this proposed method for the numerical approximation of the heat equation on moving domains. However, other problems, e.g., convection-diffusion-reaction or even incompressible fluid problems using stabilisation techniques, could be analysed using similar arguments. Exploiting the tensor-product structure of the space, we can prove optimal error estimates. In addition, we have carried out a set of numerical experiments that support the theoretical results for the heat equation and an interface mass transfer problem that involves the advection-diffusion equation. The method proves to be robust and accurate in all scenarios.

\par The present work can readily be applied to (parallel) locally refined $n$-tree meshes using the space discrete extension operator in \cite{Badia2021parallel}. The extension to adaptive mesh refinement in space and time (local time stepping) can be considered in the future.



\section*{Acknowledgments}
This research was partially funded by the Australian Government through the Australian Research Council (project number DP210103092). F. Verdugo acknowledges support from the ``Severo Ochoa Program for Centers of Excellence in R\&D (2019-2023)'' under the grant CEX2018-000797-S funded by the Ministerio de Ciencia e Innovación (MCIN) -- Agencia Estatal de Investigación (AEI/10.13039/501100011033).

\setlength{\bibsep}{0.0ex plus 0.00ex} 
\bibliographystyle{myabbrvnat}
\bibliography{refs}  

\begin{thebibliography}{57}
\providecommand{\natexlab}[1]{#1}
\providecommand{\url}[1]{\texttt{#1}}
\expandafter\ifx\csname urlstyle\endcsname\relax
  \providecommand{\doi}[1]{doi: #1}\else
  \providecommand{\doi}{doi: \begingroup \urlstyle{rm}\Url}\fi

\bibitem[Badia and Verdugo(2020)]{Badia2020gridap}
S.~Badia and F.~Verdugo.
\newblock Gridap: An extensible finite element toolbox in {J}ulia.
\newblock \emph{Journal of Open Source Software}, 5\penalty0 (52):\penalty0
  2520, Aug. 2020.
\newblock \doi{10.21105/joss.02520}.
\newblock URL \url{https://doi.org/10.21105/joss.02520}.

\bibitem[Badia et~al.(2018{\natexlab{a}})Badia, Martin, and
  Verdugo]{Badia2018mixed}
S.~Badia, A.~F. Martin, and F.~Verdugo.
\newblock Mixed aggregated finite element methods for the unfitted
  discretization of the {S}tokes problem.
\newblock \emph{{SIAM} Journal on Scientific Computing}, 40\penalty0
  (6):\penalty0 B1541--B1576, Jan. 2018{\natexlab{a}}.
\newblock \doi{10.1137/18m1185624}.
\newblock URL \url{https://doi.org/10.1137/18m1185624}.

\bibitem[Badia et~al.(2018{\natexlab{b}})Badia, Verdugo, and
  Mart{\'{\i}}n]{Badia2018aggregated}
S.~Badia, F.~Verdugo, and A.~F. Mart{\'{\i}}n.
\newblock The aggregated unfitted finite element method for elliptic problems.
\newblock \emph{Computer Methods in Applied Mechanics and Engineering},
  336:\penalty0 533--553, July 2018{\natexlab{b}}.
\newblock \doi{10.1016/j.cma.2018.03.022}.
\newblock URL \url{https://doi.org/10.1016/j.cma.2018.03.022}.

\bibitem[Badia et~al.(2021{\natexlab{a}})Badia, Hampton, and
  Principe]{Badia2021monte}
S.~Badia, J.~Hampton, and J.~Principe.
\newblock Embedded multilevel {M}onte {C}arlo for uncertainty quantification in
  random domains.
\newblock \emph{International Journal for Uncertainty Quantification},
  11\penalty0 (1):\penalty0 119--142, 2021{\natexlab{a}}.
\newblock \doi{10.1615/int.j.uncertaintyquantification.2021032984}.
\newblock URL
  \url{https://doi.org/10.1615/int.j.uncertaintyquantification.2021032984}.

\bibitem[Badia et~al.(2021{\natexlab{b}})Badia, Mart{\'{\i}}n, Neiva, and
  Verdugo]{Badia2021parallel}
S.~Badia, A.~F. Mart{\'{\i}}n, E.~Neiva, and F.~Verdugo.
\newblock The aggregated unfitted finite element method on parallel tree-based
  adaptive meshes.
\newblock \emph{{SIAM} Journal on Scientific Computing}, 43\penalty0
  (3):\penalty0 C203--C234, Jan. 2021{\natexlab{b}}.
\newblock \doi{10.1137/20m1344512}.
\newblock URL \url{https://doi.org/10.1137/20m1344512}.

\bibitem[Badia et~al.(2022{\natexlab{a}})Badia, Martorell, and
  Verdugo]{Martorell}
S.~Badia, P.~A. Martorell, and F.~Verdugo.
\newblock Geometrical discretisations for unfitted finite elements on explicit
  boundary representations.
\newblock \emph{Journal of Computational Physics}, 460:\penalty0 111162, July
  2022{\natexlab{a}}.
\newblock \doi{10.1016/j.jcp.2022.111162}.
\newblock URL \url{https://doi.org/10.1016/j.jcp.2022.111162}.

\bibitem[Badia et~al.(2022{\natexlab{b}})Badia, Neiva, and
  Verdugo]{Badia2022ghost}
S.~Badia, E.~Neiva, and F.~Verdugo.
\newblock Linking ghost penalty and aggregated unfitted methods.
\newblock \emph{Computer Methods in Applied Mechanics and Engineering},
  388:\penalty0 114232, Jan. 2022{\natexlab{b}}.
\newblock \doi{10.1016/j.cma.2021.114232}.
\newblock URL \url{https://doi.org/10.1016/j.cma.2021.114232}.

\bibitem[Badia et~al.(2022{\natexlab{c}})Badia, Neiva, and
  Verdugo]{Badia2022robust}
S.~Badia, E.~Neiva, and F.~Verdugo.
\newblock Robust high-order unfitted finite elements by interpolation-based
  discrete extension, 2022{\natexlab{c}}.
\newblock URL \url{https://arxiv.org/abs/2201.06632}.

\bibitem[Bassi et~al.(2012)Bassi, Botti, Colombo, and
  Rebay]{Bassi2012agglomeration}
F.~Bassi, L.~Botti, A.~Colombo, and S.~Rebay.
\newblock Agglomeration based discontinuous {G}alerkin discretization of the
  {E}uler and {N}avier{\textendash}{S}tokes equations.
\newblock \emph{Computers {\&} Fluids}, 61:\penalty0 77--85, May 2012.
\newblock \doi{10.1016/j.compfluid.2011.11.002}.
\newblock URL \url{https://doi.org/10.1016/j.compfluid.2011.11.002}.

\bibitem[Beau et~al.(1993)Beau, Ray, Aliabadi, and Tezduyar]{LeBeau1993}
G.~L. Beau, S.~Ray, S.~Aliabadi, and T.~Tezduyar.
\newblock {SUPG} finite element computation of compressible flows with the
  entropy and conservation variables formulations.
\newblock \emph{Computer Methods in Applied Mechanics and Engineering},
  104\penalty0 (3):\penalty0 397--422, May 1993.
\newblock \doi{10.1016/0045-7825(93)90033-t}.
\newblock URL \url{https://doi.org/10.1016/0045-7825(93)90033-t}.

\bibitem[Bezanson et~al.(2017)Bezanson, Edelman, Karpinski, and
  Shah]{Bezanson2017}
J.~Bezanson, A.~Edelman, S.~Karpinski, and V.~B. Shah.
\newblock Julia: A fresh approach to numerical computing.
\newblock \emph{{SIAM} Review}, 59\penalty0 (1):\penalty0 65--98, Jan. 2017.
\newblock \doi{10.1137/141000671}.
\newblock URL \url{https://doi.org/10.1137/141000671}.

\bibitem[Brenner and Scott(2008)]{Brenner2008}
S.~C. Brenner and L.~R. Scott.
\newblock \emph{The Mathematical Theory of Finite Element Methods}.
\newblock Springer New York, 2008.
\newblock \doi{10.1007/978-0-387-75934-0}.
\newblock URL \url{https://doi.org/10.1007/978-0-387-75934-0}.

\bibitem[Burman(2010)]{Burman2010ghost}
E.~Burman.
\newblock Ghost penalty.
\newblock \emph{Comptes Rendus Mathematique}, 348\penalty0 (21-22):\penalty0
  1217--1220, Nov. 2010.
\newblock \doi{10.1016/j.crma.2010.10.006}.
\newblock URL \url{https://doi.org/10.1016/j.crma.2010.10.006}.

\bibitem[Burman and Fern{\'{a}}ndez(2014)]{Burman2014}
E.~Burman and M.~A. Fern{\'{a}}ndez.
\newblock An unfitted {N}itsche method for incompressible
  fluid{\textendash}structure interaction using overlapping meshes.
\newblock \emph{Computer Methods in Applied Mechanics and Engineering},
  279:\penalty0 497--514, Sept. 2014.
\newblock \doi{10.1016/j.cma.2014.07.007}.
\newblock URL \url{https://doi.org/10.1016/j.cma.2014.07.007}.

\bibitem[Burman et~al.(2014)Burman, Claus, Hansbo, Larson, and
  Massing]{Burman2014cutfem}
E.~Burman, S.~Claus, P.~Hansbo, M.~G. Larson, and A.~Massing.
\newblock {CutFEM}: Discretizing geometry and partial differential equations.
\newblock \emph{International Journal for Numerical Methods in Engineering},
  104\penalty0 (7):\penalty0 472--501, Dec. 2014.
\newblock \doi{10.1002/nme.4823}.
\newblock URL \url{https://doi.org/10.1002/nme.4823}.

\bibitem[Burman et~al.(2020)Burman, Hansbo, and Larson]{Burman2020explicit}
E.~Burman, P.~Hansbo, and M.~G. Larson.
\newblock Explicit time stepping for the wave equation using {CutFEM} with
  discrete extension, 2020.
\newblock URL \url{https://arxiv.org/abs/2011.05386}.

\bibitem[Burman et~al.(2021)Burman, Cicuttin, Delay, and Ern]{Burman2021}
E.~Burman, M.~Cicuttin, G.~Delay, and A.~Ern.
\newblock An unfitted hybrid high-order method with cell agglomeration for
  elliptic interface problems.
\newblock \emph{{SIAM} Journal on Scientific Computing}, 43\penalty0
  (2):\penalty0 A859--A882, Jan. 2021.
\newblock \doi{10.1137/19m1285901}.
\newblock URL \url{https://doi.org/10.1137/19m1285901}.

\bibitem[Carraturo et~al.(2020)Carraturo, Jomo, Kollmannsberger, Reali,
  Auricchio, and Rank]{Carraturo2020additive}
M.~Carraturo, J.~Jomo, S.~Kollmannsberger, A.~Reali, F.~Auricchio, and E.~Rank.
\newblock Modeling and experimental validation of an immersed thermo-mechanical
  part-scale analysis for laser powder bed fusion processes.
\newblock \emph{Additive Manufacturing}, 36:\penalty0 101498, Dec. 2020.
\newblock \doi{10.1016/j.addma.2020.101498}.
\newblock URL \url{https://doi.org/10.1016/j.addma.2020.101498}.

\bibitem[Chrysafinos and Walkington(2006)]{Chrysafinos2006error}
K.~Chrysafinos and N.~J. Walkington.
\newblock Error estimates for the discontinuous {G}alerkin methods for
  parabolic equations.
\newblock \emph{{SIAM} Journal on Numerical Analysis}, 44\penalty0
  (1):\penalty0 349--366, Jan. 2006.
\newblock \doi{10.1137/030602289}.
\newblock URL \url{https://doi.org/10.1137/030602289}.

\bibitem[Claus and Kerfriden(2019)]{Claus2019cutfem}
S.~Claus and P.~Kerfriden.
\newblock A {CutFEM} method for two-phase flow problems.
\newblock \emph{Computer Methods in Applied Mechanics and Engineering},
  348:\penalty0 185--206, May 2019.
\newblock \doi{10.1016/j.cma.2019.01.009}.
\newblock URL \url{https://doi.org/10.1016/j.cma.2019.01.009}.

\bibitem[Dekker et~al.(2019)Dekker, Meer, Maljaars, and Sluys]{Dekker2019xfem}
R.~Dekker, F.~Meer, J.~Maljaars, and L.~Sluys.
\newblock A cohesive {XFEM} model for simulating fatigue crack growth under
  mixed-mode loading and overloading.
\newblock \emph{International Journal for Numerical Methods in Engineering},
  118\penalty0 (10):\penalty0 561--577, Feb. 2019.
\newblock \doi{10.1002/nme.6026}.
\newblock URL \url{https://doi.org/10.1002/nme.6026}.

\bibitem[Donea et~al.(1982)Donea, Giuliani, and Halleux]{Donea1982}
J.~Donea, S.~Giuliani, and J.~Halleux.
\newblock An arbitrary {L}agrangian-{E}ulerian finite element method for
  transient dynamic fluid-structure interactions.
\newblock \emph{Computer Methods in Applied Mechanics and Engineering},
  33\penalty0 (1-3):\penalty0 689--723, Sept. 1982.
\newblock \doi{10.1016/0045-7825(82)90128-1}.
\newblock URL \url{https://doi.org/10.1016/0045-7825(82)90128-1}.

\bibitem[Engwer and Heimann(2012)]{Engwer2012}
C.~Engwer and F.~Heimann.
\newblock Dune-{UDG}: A cut-cell framework for unfitted discontinuous
  {G}alerkin methods.
\newblock In \emph{Advances in {DUNE}}, pages 89--100. Springer Berlin
  Heidelberg, 2012.
\newblock \doi{10.1007/978-3-642-28589-9_7}.
\newblock URL \url{https://doi.org/10.1007/978-3-642-28589-9_7}.

\bibitem[Ern and Guermond(2004)]{Ern2004}
A.~Ern and J.-L. Guermond.
\newblock \emph{Theory and Practice of Finite Elements}.
\newblock Springer New York, 2004.
\newblock \doi{10.1007/978-1-4757-4355-5}.
\newblock URL \url{https://doi.org/10.1007/978-1-4757-4355-5}.

\bibitem[Formaggia et~al.(2021)Formaggia, Gatti, and Zonca]{Formaggia2021}
L.~Formaggia, F.~Gatti, and S.~Zonca.
\newblock An {XFEM}/{DG} approach for fluid-structure interaction problems with
  contact.
\newblock \emph{Applications of Mathematics}, 66\penalty0 (2):\penalty0
  183--211, Jan. 2021.
\newblock \doi{10.21136/am.2021.0310-19}.
\newblock URL \url{https://doi.org/10.21136/am.2021.0310-19}.

\bibitem[Giovanardi et~al.(2017)Giovanardi, Formaggia, Scotti, and
  Zunino]{Giovanardi2017}
B.~Giovanardi, L.~Formaggia, A.~Scotti, and P.~Zunino.
\newblock Unfitted {FEM} for modelling the interaction of multiple fractures in
  a poroelastic medium.
\newblock In \emph{Lecture Notes in Computational Science and Engineering},
  pages 331--352. Springer International Publishing, 2017.
\newblock \doi{10.1007/978-3-319-71431-8_11}.
\newblock URL \url{https://doi.org/10.1007/978-3-319-71431-8_11}.

\bibitem[Gross and Reusken(2011)]{Gross2011}
S.~Gross and A.~Reusken.
\newblock \emph{Numerical Methods for Two-phase Incompressible Flows}.
\newblock Springer Berlin Heidelberg, 2011.
\newblock \doi{10.1007/978-3-642-19686-7}.
\newblock URL \url{https://doi.org/10.1007/978-3-642-19686-7}.

\bibitem[Guzm{\'{a}}n et~al.(2017)Guzm{\'{a}}n, S{\'{a}}nchez, and
  Sarkis]{Guzmn2017}
J.~Guzm{\'{a}}n, M.~A. S{\'{a}}nchez, and M.~Sarkis.
\newblock A finite element method for high-contrast interface problems with
  error estimates independent of contrast.
\newblock \emph{Journal of Scientific Computing}, 73\penalty0 (1):\penalty0
  330--365, Mar. 2017.
\newblock \doi{10.1007/s10915-017-0415-x}.
\newblock URL \url{https://doi.org/10.1007/s10915-017-0415-x}.

\bibitem[Heimann(2021)]{heimann_ma}
F.~Heimann.
\newblock On discontinuous- and continuous-in-time unfitted space-time methods
  for {PDEs} on moving domains, 2021.
\newblock URL
  \url{https://data.goettingen-research-online.de/citation?persistentId=doi:10.25625/CDCMYT}.

\bibitem[Heimann et~al.(2022)Heimann, Lehrenfeld, and Preu{\ss}]{preuss}
F.~Heimann, C.~Lehrenfeld, and J.~Preu{\ss}.
\newblock Geometrically higher order unfitted space-time methods for {PDEs} on
  moving domains, 2022.
\newblock URL \url{https://arxiv.org/abs/2202.02216}.

\bibitem[Kummer(2016)]{Kummer2016}
F.~Kummer.
\newblock Extended discontinuous {G}alerkin methods for two-phase flows: the
  spatial discretization.
\newblock \emph{International Journal for Numerical Methods in Engineering},
  109\penalty0 (2):\penalty0 259--289, June 2016.
\newblock \doi{10.1002/nme.5288}.
\newblock URL \url{https://doi.org/10.1002/nme.5288}.

\bibitem[Lehrenfeld(2015)]{Lehrenfeld:462743}
C.~Lehrenfeld.
\newblock \emph{{O}n a {S}pace-{T}ime {E}xtended {F}inite {E}lement {M}ethod
  for the {S}olution of a {C}lass of {T}wo-{P}hase {M}ass {T}ransport
  {P}roblems}.
\newblock Dissertation, RWTH Aachen University, Aachen, 2015.
\newblock URL \url{https://publications.rwth-aachen.de/record/462743}.
\newblock Aachen, Techn. Hochsch., Diss., 2015.

\bibitem[Lehrenfeld(2016)]{Lehrenfeld2016}
C.~Lehrenfeld.
\newblock High order unfitted finite element methods on level set domains using
  isoparametric mappings.
\newblock \emph{Computer Methods in Applied Mechanics and Engineering},
  300:\penalty0 716--733, Mar. 2016.
\newblock \doi{10.1016/j.cma.2015.12.005}.
\newblock URL \url{https://doi.org/10.1016/j.cma.2015.12.005}.

\bibitem[Lehrenfeld and Olshanskii(2019)]{Lehrenfeld2019}
C.~Lehrenfeld and M.~Olshanskii.
\newblock An {E}ulerian finite element method for {PDEs} in time-dependent
  domains.
\newblock \emph{{ESAIM}: Mathematical Modelling and Numerical Analysis},
  53\penalty0 (2):\penalty0 585--614, Mar. 2019.
\newblock \doi{10.1051/m2an/2018068}.
\newblock URL \url{https://doi.org/10.1051/m2an/2018068}.

\bibitem[Lehrenfeld and Reusken(2013)]{Lehrenfeld2013}
C.~Lehrenfeld and A.~Reusken.
\newblock Analysis of a nitsche {XFEM}-{DG} discretization for a class of
  two-phase mass transport problems.
\newblock \emph{{SIAM} Journal on Numerical Analysis}, 51\penalty0
  (2):\penalty0 958--983, Jan. 2013.
\newblock \doi{10.1137/120875260}.
\newblock URL \url{https://doi.org/10.1137/120875260}.

\bibitem[Li et~al.(2019)Li, Atallah, Main, and Scovazzi]{Li2019}
K.~Li, N.~M. Atallah, G.~A. Main, and G.~Scovazzi.
\newblock The shifted interface method: A flexible approach to embedded
  interface computations.
\newblock \emph{International Journal for Numerical Methods in Engineering},
  121\penalty0 (3):\penalty0 492--518, Oct. 2019.
\newblock \doi{10.1002/nme.6231}.
\newblock URL \url{https://doi.org/10.1002/nme.6231}.

\bibitem[Lundholm(2015)]{lundholm2015space}
C.~Lundholm.
\newblock A space-time cut finite element method for a time-dependent parabolic
  model problem.
\newblock Master's thesis, 2015.

\bibitem[Lundholm(2021)]{lundholm2021cut}
C.~Lundholm.
\newblock \emph{Cut Finite Element Methods on Overlapping Meshes: Analysis and
  Applications}.
\newblock Chalmers Tekniska Hogskola (Sweden), 2021.

\bibitem[Martin et~al.(2022)Martin, Verdugo, Badia, and
  Colomés]{alberto_f_martin_2022_6076710}
A.~F. Martin, F.~Verdugo, S.~Badia, and O.~Colomés.
\newblock gridap/{G}ridap{D}istributed.jl: v0.2.5, Feb. 2022.
\newblock URL \url{https://doi.org/10.5281/zenodo.6076710}.

\bibitem[Neiva and Badia(2021)]{Neiva2021}
E.~Neiva and S.~Badia.
\newblock Robust and scalable h-adaptive aggregated unfitted finite elements
  for interface elliptic problems.
\newblock \emph{Computer Methods in Applied Mechanics and Engineering},
  380:\penalty0 113769, July 2021.
\newblock \doi{10.1016/j.cma.2021.113769}.
\newblock URL \url{https://doi.org/10.1016/j.cma.2021.113769}.

\bibitem[Neiva et~al.(2020)Neiva, Chiumenti, Cervera, Salsi, Piscopo, Badia,
  Mart{\'{\i}}n, Chen, Lee, and Davies]{Neiva2020}
E.~Neiva, M.~Chiumenti, M.~Cervera, E.~Salsi, G.~Piscopo, S.~Badia, A.~F.
  Mart{\'{\i}}n, Z.~Chen, C.~Lee, and C.~Davies.
\newblock Numerical modelling of heat transfer and experimental validation in
  powder-bed fusion with the virtual domain approximation.
\newblock \emph{Finite Elements in Analysis and Design}, 168:\penalty0 103343,
  Jan. 2020.
\newblock \doi{10.1016/j.finel.2019.103343}.
\newblock URL \url{https://doi.org/10.1016/j.finel.2019.103343}.

\bibitem[Nitsche(1971)]{Nitsche1971}
J.~Nitsche.
\newblock \"{U}ber ein {V}ariationsprinzip zur {L}\"{o}sung von
  {D}irichlet-{P}roblemen bei {V}erwendung von {T}eilr\"{a}umen, die keinen
  {R}andbedingungen unterworfen sind.
\newblock \emph{Abhandlungen aus dem Mathematischen Seminar der Universit\"{a}t
  Hamburg}, 36\penalty0 (1):\penalty0 9--15, July 1971.
\newblock \doi{10.1007/bf02995904}.
\newblock URL \url{https://doi.org/10.1007/bf02995904}.

\bibitem[Nobile and Formaggia(1999)]{nobile1999stability}
F.~Nobile and L.~Formaggia.
\newblock A stability analysis for the arbitrary {L}agrangian {E}ulerian
  formulation with finite elements.
\newblock \emph{East-West Journal of Numerical Mathematics}, 7\penalty0
  (2):\penalty0 105--132, 1999.

\bibitem[Preu{\ss}(2021)]{preuss_ma}
J.~Preu{\ss}.
\newblock "{H}igher order unfitted isoparametric space-time {FEM} on moving
  domains" (master's thesis), 2021.
\newblock URL
  \url{https://data.goettingen-research-online.de/citation?persistentId=doi:10.25625/UACWXS}.

\bibitem[Reusken(2014)]{Reusken2014}
A.~Reusken.
\newblock Analysis of trace finite element methods for surface partial
  differential equations.
\newblock \emph{{IMA} Journal of Numerical Analysis}, 35\penalty0 (4):\penalty0
  1568--1590, Oct. 2014.
\newblock \doi{10.1093/imanum/dru047}.
\newblock URL \url{https://doi.org/10.1093/imanum/dru047}.

\bibitem[Saito(2020)]{Saito2020}
N.~Saito.
\newblock Variational analysis of the discontinuous {G}alerkin time-stepping
  method for parabolic equations.
\newblock \emph{{IMA} Journal of Numerical Analysis}, 41\penalty0 (2):\penalty0
  1267--1292, May 2020.
\newblock \doi{10.1093/imanum/draa017}.
\newblock URL \url{https://doi.org/10.1093/imanum/draa017}.

\bibitem[Saye(2017)]{Saye2017}
R.~Saye.
\newblock Implicit mesh discontinuous {G}alerkin methods and interfacial gauge
  methods for high-order accurate interface dynamics, with applications to
  surface tension dynamics, rigid body fluid{\textendash}structure interaction,
  and free surface flow: Part {I}.
\newblock \emph{Journal of Computational Physics}, 344:\penalty0 647--682,
  Sept. 2017.
\newblock \doi{10.1016/j.jcp.2017.04.076}.
\newblock URL \url{https://doi.org/10.1016/j.jcp.2017.04.076}.

\bibitem[Schott et~al.(2019)Schott, Ager, and Wall]{Schott2019}
B.~Schott, C.~Ager, and W.~A. Wall.
\newblock Monolithic cut finite~element{\textendash}based approaches for
  fluid-structure interaction.
\newblock \emph{International Journal for Numerical Methods in Engineering},
  119\penalty0 (8):\penalty0 757--796, Apr. 2019.
\newblock \doi{10.1002/nme.6072}.
\newblock URL \url{https://doi.org/10.1002/nme.6072}.

\bibitem[Smears(2016)]{Smears2016}
I.~Smears.
\newblock Robust and efficient preconditioners for the discontinuous {G}alerkin
  time-stepping method.
\newblock \emph{{IMA} Journal of Numerical Analysis}, page drw050, Oct. 2016.
\newblock \doi{10.1093/imanum/drw050}.
\newblock URL \url{https://doi.org/10.1093/imanum/drw050}.

\bibitem[Sudirham et~al.(2006)Sudirham, van~der Vegt, and van
  Damme]{Sudirham2006}
J.~Sudirham, J.~van~der Vegt, and R.~van Damme.
\newblock Space{\textendash}time discontinuous {G}alerkin method for
  advection{\textendash}diffusion problems on time-dependent domains.
\newblock \emph{Applied Numerical Mathematics}, 56\penalty0 (12):\penalty0
  1491--1518, Dec. 2006.
\newblock \doi{10.1016/j.apnum.2005.11.003}.
\newblock URL \url{https://doi.org/10.1016/j.apnum.2005.11.003}.

\bibitem[Tezduyar et~al.(2006)Tezduyar, Sathe, Keedy, and Stein]{Tezduyar2006}
T.~E. Tezduyar, S.~Sathe, R.~Keedy, and K.~Stein.
\newblock Space{\textendash}time finite element techniques for computation of
  fluid{\textendash}structure interactions.
\newblock \emph{Computer Methods in Applied Mechanics and Engineering},
  195\penalty0 (17-18):\penalty0 2002--2027, Mar. 2006.
\newblock \doi{10.1016/j.cma.2004.09.014}.
\newblock URL \url{https://doi.org/10.1016/j.cma.2004.09.014}.

\bibitem[Thom{\'{e}}e()]{Thome}
V.~Thom{\'{e}}e.
\newblock The discontinuous {G}alerkin time stepping method.
\newblock In \emph{Galerkin Finite Element Methods for Parabolic Problems},
  pages 203--230. Springer Berlin Heidelberg.
\newblock \doi{10.1007/3-540-33122-0_12}.
\newblock URL \url{https://doi.org/10.1007/3-540-33122-0_12}.

\bibitem[Thompson and Pinsky(1996)]{Thompson1996}
L.~L. Thompson and P.~M. Pinsky.
\newblock A space-time finite element method for structural acoustics in
  infinite domains part 1: Formulation, stability and convergence.
\newblock \emph{Computer Methods in Applied Mechanics and Engineering},
  132\penalty0 (3-4):\penalty0 195--227, June 1996.
\newblock \doi{10.1016/0045-7825(95)00955-8}.
\newblock URL \url{https://doi.org/10.1016/0045-7825(95)00955-8}.

\bibitem[Verdugo and Badia(2022)]{Verdugo2022}
F.~Verdugo and S.~Badia.
\newblock The software design of {G}ridap: A finite element package based on
  the {J}ulia {JIT} compiler.
\newblock \emph{Computer Physics Communications}, 276:\penalty0 108341, July
  2022.
\newblock \doi{10.1016/j.cpc.2022.108341}.
\newblock URL \url{https://doi.org/10.1016/j.cpc.2022.108341}.

\bibitem[Verdugo et~al.(2019)Verdugo, Mart{\'{\i}}n, and Badia]{Verdugo2019}
F.~Verdugo, A.~F. Mart{\'{\i}}n, and S.~Badia.
\newblock Distributed-memory parallelization of the aggregated unfitted finite
  element method.
\newblock \emph{Computer Methods in Applied Mechanics and Engineering},
  357:\penalty0 112583, Dec. 2019.
\newblock \doi{10.1016/j.cma.2019.112583}.
\newblock URL \url{https://doi.org/10.1016/j.cma.2019.112583}.

\bibitem[Verdugo et~al.(2021)Verdugo, Neiva, and Badia]{GridapEmbedded-jl}
F.~Verdugo, E.~Neiva, and S.~Badia.
\newblock {GridapEmbedded. Version 0.7.}, Oct. 2021.
\newblock URL \url{https://github.com/gridap/GridapEmbedded.jl}.
\newblock Available at \url{https://github.com/gridap/GridapEmbedded.jl}.

\bibitem[Zahedi(2017)]{Zahedi2017}
S.~Zahedi.
\newblock A space-time cut finite element method with quadrature in time.
\newblock In \emph{Lecture Notes in Computational Science and Engineering},
  pages 281--306. Springer International Publishing, 2017.
\newblock \doi{10.1007/978-3-319-71431-8_9}.
\newblock URL \url{https://doi.org/10.1007/978-3-319-71431-8_9}.

\end{thebibliography}

\end{document}